\documentclass{daj}

\usepackage{amsmath,amssymb,amsthm, amscd, mathtools}
\usepackage{pifont,mathrsfs}
\usepackage{array,lastpage}
\usepackage{enumerate,xspace,pifont,shadow}
\usepackage{mathrsfs}
\usepackage[T1]{fontenc}
\usepackage{mathtools}
\usepackage{caption}
\usepackage{subcaption}

\DeclareSymbolFont{AMSb}{U}{msb}{m}{n}
\DeclareMathSymbol{\Z}{\mathbin}{AMSb}{"5A}
\DeclareMathSymbol{\R}{\mathbin}{AMSb}{"52}
\DeclareMathSymbol{\N}{\mathbin}{AMSb}{"4E}
\DeclareMathSymbol{\Q}{\mathbin}{AMSb}{"51}

\DeclareMathOperator{\conv}{conv}
\DeclareMathOperator{\cone}{cone}
\DeclareMathOperator{\spanOp}{span}
\DeclareMathOperator{\zono}{zono}

\newcommand{\A}{\mathbf{a}}
\newcommand{\B}{\mathbf{b}}

\newcommand{\V}{\mathbf{v}}
\newcommand{\U}{\mathbf{u}}
\newcommand{\w}{\mathbf{w}}
\newcommand{\x}{\mathbf{x}}
\newcommand{\y}{\mathbf{y}}
\newcommand{\z}{\mathbf{z}}

\def\vec#1{\mathchoice{\mbox{\boldmath$\displaystyle\bf#1$}}
{\mbox{\boldmath$\textstyle\bf#1$}}
{\mbox{\boldmath$\scriptstyle\bf#1$}}
{\mbox{\boldmath$\scriptscriptstyle\bf#1$}}}

\providecommand{\abs}[1]{\lvert#1\rvert}
\providecommand{\floor}[1]{\left\lfloor#1\right\rfloor}

\newtheorem{thm}{Theorem}[section]
\newtheorem{lem}[thm]{Lemma}
\newtheorem{cor}[thm]{Corollary}
\newtheorem{prop}[thm]{Proposition}

\newtheorem{ass}[thm]{Assumption}

\theoremstyle{definition}
\newtheorem{definition}[thm]{Definition}

\theoremstyle{remark}
\newtheorem{remark}[thm]{Remark}

\theoremstyle{remark}
\newtheorem{example}[thm]{Example}

\theoremstyle{remark}

\theoremstyle{remark}

\theoremstyle{remark}

\theoremstyle{remark}

\dajAUTHORdetails{%
  title = {Parametric Presburger Arithmetic: Logic, Combinatorics, and Quasi-polynomial Behavior}, 
  author = {Tristram Bogart, John Goodrick, and Kevin Woods},
  plaintextauthor = {Tristram Bogart, John Goodrick, and Kevin Woods},
  keywords = {Presburger arithmetic, quasi-polynomials, lattice points, Ehrhart polynomials}
}   

\dajEDITORdetails{%
   year={2017},
   number={4},
   received={13 September 2016},   
   published={16 January 2017},  
   doi={10.19086/da.1254},       
}   

\begin{document}

\begin{frontmatter}[classification=text]

\title{Parametric Presburger Arithmetic:\\Logic, Combinatorics, and\\Quasi-polynomial Behavior} 

\author[tbog]{Tristram Bogart\thanks{Partially supported by the Fondo de Apoyo a Profesores Asistentes (FAPA) of the Universidad de los Andes, Colombia}}
\author[jgoo]{John Goodrick\thanks{Partially supported by the Fondo de Apoyo a Profesores Asistentes (FAPA) of the Universidad de los Andes, Colombia}}
\author[kwoo]{Kevin Woods}

\begin{abstract}
Parametric Presburger arithmetic concerns families of sets $S_t\subseteq\Z^d$, for $t\in\N$, that are defined using addition, inequalities, constants in $\Z$, Boolean operations, multiplication by $t$, and quantifiers on variables ranging over $\Z$. That is, such families are defined using quantifiers and Boolean combinations of formulas of the form $\vec a(t)\cdot \x \le b(t),$ where $\vec a(t)\in\Z[t]^d, b(t)\in\Z[t]$. A function $g:\N\rightarrow\Z$ is a quasi-polynomial if there exists a period $m$ and polynomials $f_0,\ldots,f_{m-1}\in\Q[t]$ such that $g(t)=f_i(t),\text{ for }t\equiv i\bmod m.$

  Recent results of Chen, Li, Sam; Calegari, Walker; Roune, Woods; and Shen concern specific families in parametric Presburger arithmetic that exhibit quasi-polynomial behavior. For example, $\abs{S_t}$ might be a quasi-polynomial function of $t$ or an element $\x(t)\in S_t$ might be specifiable as a function with quasi-polynomial coordinates, for sufficiently large $t$. Woods conjectured that all parametric Presburger sets exhibit this quasi-polynomial behavior. Here, we prove this conjecture, using various tools from logic and combinatorics.\end{abstract}
\end{frontmatter}

\section{Introduction}

We examine a broad class of problems that exhibit \emph{quasi-polynomial} behavior.

\begin{definition}
\label{def:QP}
A function $g:\N\rightarrow\Z$ is a \emph{quasi-polynomial} if there exists a period $m$ and polynomials $f_0,\ldots,f_{m-1}\in\Q[t]$ such that
\[g(t)=f_i(t),\text{ for }t\equiv i\bmod m.\]
A function $g:\N\rightarrow\Z$ is an \emph{eventual quasi-polynomial}, abbreviated \emph{EQP}, if it agrees with a quasi-polynomial for sufficiently large $t$. (In this paper, we take $\N=\{0,1,2,\ldots\}$.)
\end{definition}

\begin{example}
\label{ex:floor}
\[g(t)=\floor{\frac{t+1}{2}}=\begin{cases} \frac{t}{2} & \text{if $t$ even},\\ \frac{t+1}{2} & \text{if $t$ odd},\end{cases}\]
is a quasi-polynomial with period 2.
\end{example}

In \cite{woods1}, Woods noticed that several recent results concern different kinds of EQP behavior in combinatorially defined families. For example, given a family of sets $S_t\subseteq\Z^d$ for $t\in\N$, $\abs{S_t}$ might be an EQP function of $t$ or an element $\x(t)\in S_t$ might be specifiable as a function with EQP coordinates. These results are all concerned with the following types of sets \cite{woods1}:

\begin{definition}
Given $d \in \N$, a \emph{parametric Presburger family} is a collection $\{S_t : t \in \N\}$ of subsets of $\Z^d$ which can be defined by a formula using addition, inequalities, multiplication and addition by constants from $\Z$, Boolean operations (and, or, not), multiplication by $t$, and quantifiers ($\forall$, $\exists$) on variables ranging over $\Z$. That is, such families are defined using quantifiers and Boolean combinations of formulas of the form $\vec a(t)\cdot \x \le b(t),$ where $\vec a(t)\in\Z[t]^d, b(t)\in\Z[t]$.
\end{definition}

In Section \ref{sub:examples}, we give a number of examples of parametric Presburger families, and discuss their previously known EQP behavior.  Note that, for fixed $t$, a formula $\vec a(t)\cdot \x \le b(t)$ is simply a linear inequality; as $t$ changes, the half-space that this linear inequality defines both shifts (as $b(t)$ changes) and rotates (as $\vec a(t)$ changes). It is important to emphasize that we do not allow quantifiers $\exists t$ or $\forall t$ applied to the parameter $t$.

In \cite{woods1}, Woods conjectured that these parametric Presburger families always have EQP behavior. In this paper, we prove this conjecture. After seeing several examples in Section \ref{sub:examples}, we state this theorem precisely in Section \ref{sub:statement}.

\subsection{Examples}
\label{sub:examples}

\begin{example}
\label{Ex:Ehrhart} Let $S_{t}$ be the set of integer points, $\x\in\Z^d$, in a parametric polyhedron defined by a conjunction of linear inequalities of the form $\vec a \cdot \vec x\le tb$, where $\vec a\in\Z^d$ and $b\in\Z$. That is, $S_t=tP\cap\Z^d$ for some rational polyhedron $P\subseteq\R^d$.
\end{example}

Ehrhart proved \cite{ehrhart} that $\abs{S_t}$ (if finite) is a quasi-polynomial, with a period given by the smallest $m$ such that $mP$ has integer vertices. This is the classic quasi-polynomial result; see the book \cite{BR07} by Beck and Robins for a proof and many examples of its utility. As a concrete example:

\begin{example}
\label{ex:triangle}
Let $P$ be the triangle with vertices $(0,0)$, $\left(\frac{1}{2},0\right)$, and $\left(\frac{1}{2},\frac{1}{2}\right)$. Then
\[g(t)=\#(tP\cap\Z^2)=\frac{\big(\floor{t/2}+1\big)\big(\floor{t/2}+2\big)}{2}=\begin{cases} (t+2)(t+4)/8&\text{if $t$ even,}\\ (t+1)(t+3)/8 & \text{if $t$ odd,}\end{cases}\]
is a quasi-polynomial with period 2.
\end{example}

More recently, Chen, Li, and Sam proved \cite{ChLiSam} that $\abs{S_t}$ is still an EQP, even if the normal vectors in the linear inequalities are allowed to vary with $t$, that is, if they are of the form $\vec a(t)\cdot \x \le b(t),$ where $\vec a(t)\in\Z[t]^d, b(t)\in\Z[t]$.

For a concrete example:

\begin{example}
\label{ex:twist}
Let $P_t$ be the ``twisting square'' in Figure 1a, defined by
\[P_t=\big\{(x,y)\in\mathbb R^2:\  \abs{2x+(2t-2)y}\le t^2-2t+2,\ \abs{(2-2t)x+2y}\le t^2-2t+2\big\}.\]
Then $\abs{P_t\cap\Z^2}$ is given by the quasi-polynomial
\[\abs{P_t\cap\Z^2}=\begin{cases}t^2-2t+2 &\text{if $t$ odd,}\\t^2-2t+5 &\text{if $t$ even}. \end{cases}
\]
\end{example}

\begin{figure}
\centering
\caption{The twisting square $P_t$ from Example \ref{ex:twist}.}
\begin{subfigure}{.4\textwidth}
\centering
\includegraphics[width=.9\textwidth]{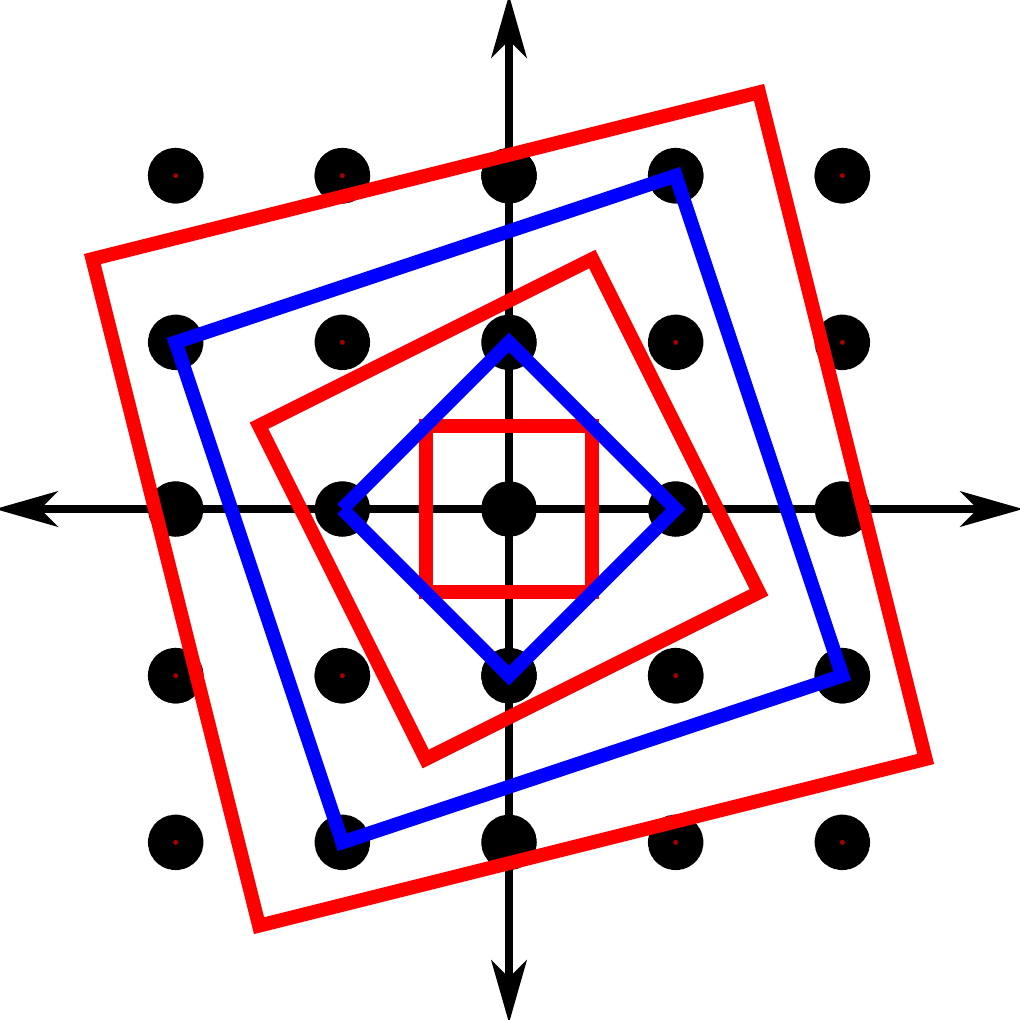}\caption{$P_t$ for $t=1,\ldots,5$.}\label{fig:twist}
\end{subfigure}
\begin{subfigure}{.4\textwidth}
\centering
\includegraphics[width=.9\textwidth]{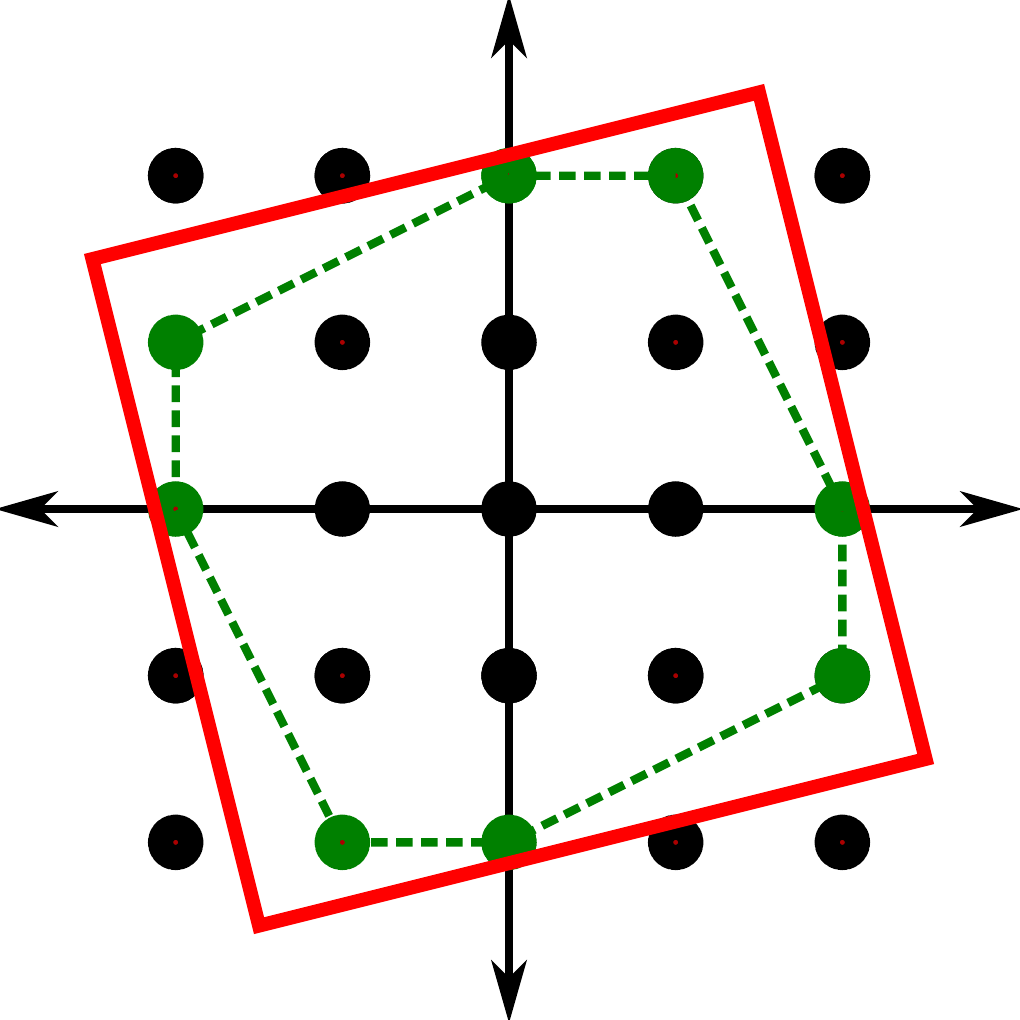}\caption{Integer hull of $P_5$.}\label{fig:twist2}
\end{subfigure}
\end{figure}

Calegari and Walker were similarly concerned \cite{CW} with the integer points  in a polyhedron, $P_t$, defined by linear inequalities of the form $\vec a(t)\cdot\x\le b(t)$. Rather than counting $\abs{P_t\cap\Z^d}$, they wanted to find the vertices of the integer hull of $P_t$, that is, the vertices of the convex hull of $P_t\cap\Z^d$.

\begin{example}
Consider the twisting square, $P_t$, from Example \ref{ex:twist}. When $t$ is even, the vertices of $P_t$ are integers, so the vertices of the integer hull are simply the vertices of $P_t$:
\[\left(\pm\frac{t-2}{2},\pm\frac{t}{2}\right)\quad\text{and}\quad \left(\pm\frac{t}{2},\mp\frac{t-2}{2}\right).\]
When $t$ is odd, the integer hull of $P_t$ is an octagon (pictured in Figure 1b for $t=5$) with vertices
\[\left(0,\pm\frac{t-1}{2}\right),\left(\pm\frac{t-3}{2},\pm\frac{t-1}{2}\right),\left(\pm\frac{t-1}{2},0\right),\left(\pm\frac{t-1}{2},\mp\frac{t-3}{2}\right).\]
\end{example}

Abstracting from the last example, the following turns out to be true (first proved by Calegari and Walker \cite{CW} under an additional hypothesis on $P_t$, later proved in general by Shen \cite{Shen15b}): 

\begin{thm}
\label{integer_hulls}
Suppose that $P_t \subseteq \R^d$ is a family of polyhedra defined by a finite conjunction of linear inequalities $\vec a(t)\cdot\x\le b(t)$ where $\vec a \in \Z[t]^d$, $b \in \Z[t]$, and that every $P_t$ is bounded. Then there exists a modulus $m$ and functions $\vec p_{ij}:\N \rightarrow\Z^d$ with polynomial coordinates such that, for $0\le i<m$ and for sufficiently large $t\equiv i \bmod m$, the set of vertices of the integer hull of $P_t \cap \Z^d$ is $\{\vec p_{i1}(t),\vec p_{i2}(t),\ldots,\vec p_{ik_i}(t)\}$.
\end{thm}

Another recent EQP result involves the Frobenius number:

\begin{definition}
Given $a_1,\ldots, a_d\in\N$, let $S$ be the semigroup generated by the $a_i$, that is,
\[S=\{a\in\N:\ \exists \lambda_1,\ldots,\lambda_d\in\N,\  a=\lambda_1a_1+\cdots+\lambda_d a_d\}.\]
Note the heavy use of quantifiers in this definition, demonstrating that parametric Presburger arithmetic is a natural setting. If the $a_i$ are relatively prime, then $S$ contains all sufficiently large integers, and the \emph{Frobenius number} is defined to be the largest integer not in $S$. This is easily encoded in Presburger arithmetic  as the (unique) element $x \in \N$ satisfying
\[ x \notin S \wedge \forall y \in \Z \left( y \notin S \rightarrow y \leq x \right). \]
\end{definition}

Now we let $a_i=a_i(t)$ vary with $t$. Roune and Woods proved \cite{roune_woods}, in a few special cases, that the Frobenius number is an EQP. More recently, Shen \cite{Shen15} proved that this is always true for any eventually positive $a_i\in\Z[t]$.

\begin{example}
Let $a_1(t)=t$, $a_2(t)=t+1$, and $a_3(t)=t+2$. Then the Frobenius number is
\[\left(\floor{\frac{t-2}{2}}+1\right)t-1.\]
This is a running example in \cite{roune_woods}.
\end{example}

One more example (from \cite{woods1}) concisely shows the sort of EQP behavior that may appear:

\begin{example}
\label{ex:PA}
Given $t\in\N$, let
\[S_t=\big\{x\in\N:\ \exists y\in\N,\  2x+2y+3=5t\text{ and }t< x\le y\big\}.\]
We can compute that
\[S_t=\begin{cases}\left\{t+1,t+2,\ldots,\floor{\frac{5t-3}{4}}\right\} & \text{if $t$ odd, $t\ge 3$,}\\
\emptyset & \text{else.}
\end{cases}\]
\end{example}

This set has several properties, which we will formalize in Section \ref{sub:statement}:
\begin{enumerate}
\item The set of $t$ such that $S_t$ is nonempty is $\{3,5,7,\ldots\}$. This set is eventually periodic.
\item The cardinality of $S_t$ is
\[|S_t|=\begin{cases}\floor{\frac{5t-3}{4}}-t & \text{if $t$ odd, $t\ge 3$,}\\
0 & \text{else,}
\end{cases}\]
which is eventually a quasi-polynomial of period 4.
\item When $S_t$ is nonempty, we can obtain an element of (indeed, the maximum element of) $S_t$ with the function $x(t)=\floor{(5t-3)/4}$, and $x(t)$ is eventually a quasi-polynomial.
\item We can compute the generating function
\begin{align*}
\sum_{s\in S_t}x^s&=\begin{cases}x^{t+1}+x^{t+2}+\cdots+x^{\floor{(5t-3)/4}} & \text{if $t$ odd, $t\ge 3$,}\\
0 & \text{else,}
\end{cases}\\
&=\begin{cases}\dfrac{x^{t+1}-x^{\floor{(5t-3)/4)}+1}}{1-x} & \text{if $t$ odd, $t\ge 3$,}\\
0 & \text{else.}
\end{cases}
\end{align*}
We see that, for fixed $t$, this generating function is a rational function. Considering each residue class of $t \bmod 4$ separately, the exponents in the rational function can eventually be written as polynomials in $t$.
\end{enumerate}

Here are two examples that show that our precise definition of parametric Presburger families is important in order to get EQP behavior:

\begin{example}
The family $$S_t = \{x \in \Z : \exists y \left[y \geq 0 \wedge xy = t \right] \}$$ is not a parametric Presburger family, because two variables (neither of which are the parameter, $t$) are multiplied together.  Indeed, $\abs{S_t}$ is the number of nonnegative divisors of $t$, which is not an EQP function in $t$.
\end{example}

\begin{example}
The family \[S_{s,t}=\big\{(x,y)\in\Z^2:\ x\ge 0\wedge y\ge 0\wedge sx+ty=st\big\}\] has two parameters, $s,t\in \N$. $S_{s,t}$ is an interval in $\Z^2$ with endpoints $(t,0)$ and $(0,s)$, and
\[\abs{S_{s,t}}=\gcd(s,t)+1.\]
For a \emph{fixed} $s$, this is an EQP function of $t$ (and vice versa), but it is not jointly an EQP function in $s$ and $t$.
\end{example}

\subsection{Statement of result}\label{sub:statement}
In general, let $S_t$, for $t\in\N$, be a family of subsets of $\Z^d$ and consider the following properties that $S_t$ might have, \textit{cf.} Example \ref{ex:PA}.

\begin{description}
\item[Property 1] The set of $t$ such that $S_t$ is nonempty is eventually periodic.
\end{description}

\begin{description}
\item[Property 2] There exists an EQP $g:\N\rightarrow\N$ such that, if $S_t$ has finite cardinality, then $g(t)=\abs{S_t}$. The set of $t$ such that $S_t$ has finite cardinality is eventually periodic.
\end{description}

\begin{description}
\item[Property 3] There exists a function $\x:\N\rightarrow\Z^d$, whose coordinate functions are EQPs, such that, if $S_t$ is non\-emp\-ty, then $\x(t)\in S_t$. The set of $t$ such that $S_t$ is nonempty is eventually periodic.
\end{description}

\begin{description}
\item[Property 4] (Assuming $S_t\subseteq\N^d$) There exists a period $m$ such that, for sufficiently large $t\equiv i\bmod m$, \[\sum_{\x\in S_t}\z^\x = \frac{\sum_{j=1}^{n_i}\alpha_{ij}\z^{\mathbf{q_{ij}}(t)}}{(1-\z^{\mathbf{b_{i1}}(t)})\cdots(1-\z^{\mathbf{b_{ik_i}}(t)})},\]
where $\alpha_{ij}\in\Q$, and the coordinate functions of $\mathbf{q_{ij}},\mathbf{b_{ij}}:\N\rightarrow\Z^d$ are polynomials with the $\mathbf{b_{ij}}(t)$ eventually lexicographically positive. (That is, for sufficiently large $t$, the vector $\mathbf{b_{ij}}(t)$ is nonzero and its first nonzero coordinate is positive.) 
\end{description}

Property 2 is about counting all solutions and Property 3 is about obtaining specific solutions, and so they seem somewhat different. Generating functions and Property 4 turn out to be the appropriate common generalization. Woods  proves \cite[Theorem 3.3]{woods1} a web of implications among these properties; in particular:  

\begin{thm} \label{implications} 
Let $S_t$ be any family of subsets of $\N^d$. If $S_t$ satisfies Property 4, then it also satisfies Properties 1, 2, and 3.
\end{thm}

Woods conjectures \cite{woods1} that these properties all hold for parametric Presburger families. The contribution of this paper is to prove this conjecture, namely: 
  
\begin{thm} \label{main}
Suppose $S_t\subseteq \Z^d$ is a parametric Presburger family. Then Properties 1, 2, and 3 all hold. Furthermore, if $S_t\subseteq \N^d$, then Property 4 holds. 
\end{thm}

In general, quantifiers and Boolean operations allow us to perform many useful operations on sets. For example, if $S_t$ is any parametric Presburger set and $\vec c \in \Z^d$ is constant, then the set of $\vec x\in S_t$ maximizing $\vec c \cdot \vec x$ can be defined 
using the parametric Presburger formula
\[\x\in S_t\quad \wedge\quad \forall y\in\Z^d \left(\y\in S_t \rightarrow \vec c\cdot \y\le \vec c\cdot \x\right).\]
Likewise, if we know that Property 3 holds for all parametric Presburger families, and if $\vec x : \N \rightarrow \Z^d$ is a function with EQP coordinates picking out an element of $S_t$ whenever this set is nonempty, then we can apply Property 3 again to the new parametric Presburger family $S_t \setminus \{\vec x (t) \}$ to obtain a function $\vec x_2 : \N \rightarrow \Z^d$ such that whenever $|S_t| \geq 2$, the pair $(\vec x(t), \vec x_2(t) )$ selects two distinct elements of $S_t$.

Arguing as in the previous paragraph, Woods proved \cite[Theorem 3.4]{woods1} that the following would be an immediate corollary to Theorem~\ref{main}:

\begin{cor}
Suppose $S_t\subseteq \Z^d$ is a parametric Presburger family. Then $S_t$ has the following properties:
\begin{description}
\item[Property  3a] Given $\vec c\in \Z^d\setminus\{0\}$, there exists a function $\vec x:\N\rightarrow\Z^d$ such that, if $\max_{\y\in S_t} \vec c\cdot \y$ exists, then it is attained at $\x(t)\in S_t$, and the coordinate functions of $\x$ are EQPs. The set of  $t$ such that the maximum exists is eventually periodic.

\item[Property  3b] Fix $k\in\N$. There exist functions $\vec x_1,\ldots,\vec x_k:\N\rightarrow\Z^d$ such that, if $\abs{S_t}\ge k$, then $\x_1(t),\ldots,\x_k(t)$ are distinct elements of $S_t$, and the coordinate functions of $\x_i$ are EQPs. The set of $t$ such that $\abs{S_t}\ge k$ is eventually periodic.
\end{description}
\end{cor}

As a further illustration of the power of parametric Presburger arithmetic, note that the ability to maximize a linear functional allows us to run the beyond-and-beneath algorithm (see Gr\"unbaum \cite[Section 5.2]{Grunbaum}) on the integer points in a parametric polyhedron, $P_t$, defined by linear inequalities of the form $\vec a(t)\cdot\x\le b(t)$. This will iteratively compute vertices of the integer hull, and they will have EQP coordinates. Given the uniform (across $t$) bound on the number of vertices, as proved in \cite{CW}, this suffices to re-prove Theorem \ref{integer_hulls}. All other theorems discussed in Section \ref{sub:examples} are immediate consequences of Theorem \ref{main}.

First, in Section  \ref{sec:outline} we give an outline of the proof of Theorem \ref{main}, which combines ideas from logic and combinatorics. We set up the logical foundations carefully in Section \ref{sec:logic}, and then we prove Properties 1, 2, and 3 for Theorem \ref{main} in Section \ref{sec:proof}. This allows us to prove them in the full generality of $\Z^d$, but more importantly it shows that the generating function tools are not necessary: we have enough combinatorial information about the sets $S_t$ to prove Properties 1, 2 and 3 directly. In Section \ref{sec:pl}, we discuss generating functions in more detail and prove Property 4. 

\section{Outline of Proof}
\label{sec:outline}

We are given a set $S_t$ defined in parametric Presburger arithmetic, whose language we will denote $\mathcal{L}_{\textup{EQP}}$. This language allows (repeated) multiplication by $t$, so that the \emph{atomic formulas} (the basic building blocks) are of the form $\vec a(t)\cdot \x \le b(t)$ or $\vec a(t)\cdot \x = b(t)$, where $\vec a(t)\in\Z[t]^d, b(t)\in\Z[t]$. More complicated formulas may be built up using Boolean operations and quantifiers.

By contrast, the language of \emph{standard} Presburger arithmetic, which we denote $\mathcal L_{Pres}$, does not allow multiplication by $t$, so that its atomic formulas are of the form $\vec a\cdot \x \le b$ or $\vec a\cdot x = b$, where $\vec a\in\Z^d, b\in\Z$. A standard technique to analyze formulas in $\mathcal L_{Pres}$ is \emph{quantifier elimination}: obtain a \emph{logically equivalent} formula (one defining the same set) that has no quantifiers.

\begin{example}
The statements
\[\exists y\left[x+y\le z \ \wedge\ y\ge 3\right]\quad\text{and} \quad x\le z-3\]
are logically equivalent modulo the theory of Presburger arithmetic; in the second statement, the quantified variable $y$ has been eliminated.
\end{example}

Presburger originally analyzed \cite{Presburger29} (see \cite{Presburger91} for a translation) these formulas; Cooper \cite{cooper} uses the following strategy for quantifier elimination:

Eliminate quantified variables one at a time (innermost to outermost). Since $\forall y\ \varphi(\x, y)$ is equivalent to $\neg \exists y\, \neg \varphi(\x, y)$, we may assume we are eliminating an existential quantifier from a formula of the form $\exists y\ \varphi(\x, y)$. Cooper's strategy is to find a finite set of candidate $y$'s (each written in terms of $\x$) such that:  there exists $y$ making $\varphi(\x, y)$ true if and only if one of the candidate $y$'s makes $\varphi(\x,y)$ true.

\begin{example}
If there exists a $y\in\Z$ such that
\[x+1\le y\le z\ \wedge\ 2y\le 3z-x\]
then $y=x+1$ must be such a $y$ (this is the smallest integer $y$ satisfying the only lower bound on $y$, and therefore it is the most likely $y$ to also satisfy the upper bounds). Substituting in $y=x+1$ eliminates the quantifier, and we are left with
\[x+1\le z\ \wedge\ 2(x+1)\le 3z-x.\]
\end{example}

In general, unfortunately, this does not quite work to guarantee that we may eliminate a quantifier. Instead, we may need to introduce some divisibility conditions:

\begin{example}
If there exists a $y\in\Z$ such that
\[x+1\le 2y\le z\ \wedge\ 2y\le 3z-x\]
then our candidate for $y$ depends on the parity of $x$: if $x$ is odd, then $y=(x+1)/2$ is our candidate, and if $x$ is even, then $y=(x+2)/2$ is our candidate. Allowing ourselves to use divisibility by 2 in the formula, we eliminate $y$ to get:
\[\big(2\big|x\ \wedge\ x+2\le z\ \wedge\ x+2\le 3z-x\big)\ \ \vee\ \  \big(2\big|(x-1)\ \wedge\ x+1\le z\ \wedge\ x+1\le 3z-x\big).\]
\end{example}

To completely eliminate all quantifies in $\mathcal L_{Pres}$, we must \emph{extend} our language to $\mathcal L^+_{Pres}$, where testing divisibility by a constant, $c$, is allowed.  Formally, we introduce \emph{divisibility predicates}, $D_c$, for constants $c$, into our language. For example, we write $2\big|(x-1)$ as $D_2(x-1)$.

Our proof begins by trying  to apply these techniques in $\mathcal L_{\textup{EQP}}$, where multiplication by $t$ is allowed.

\bigskip

\textbf{Step 1:} Goodrick does exactly this in \cite{presburger_bounded_qe} (and Lasaruk and Sturm \cite{lasaruksturmweakQE} independently arrived at essentially the same result using a different technique). Not surprisingly, we need to extend our language to $\mathcal{L}^+_{\textup{EQP}}$, where we are also allowed to test divisibility by polynomials in $t$ (that is, we allow divisibility predicates $D_{f(t)}$, where $f(t)$ is a polynomial in $t$). Even this is not quite enough to eliminate quantifiers:

\begin{example}
\label{Ex:step1}
If there exists a $y\in\Z$ such that
\[x+1\le ty\le z\ \wedge\ ty\le 3z-x\]
then our candidate for $y$ depends on $x\bmod t$: if $x\equiv i\bmod t$, for $0\le i\le t-1$, then $y=(x+t-i)/t$ is our candidate. Unfortunately, we now have $t$ different candidates for $y$, and we cannot simply list them in a formula. Instead, we must write something like:
\[\exists i \left[0\le i\le t-1\ \wedge\ t\big| (x-i)\ \wedge\ \big(x+t-i\le z\big)\ \wedge\ \big(x+t-i\le 3z-x \big)\right]\]
We have replaced the old quantified $y$ with a new quantified $i$, but we have gained something: $i$ is \emph{bounded} by a polynomial in $t$.
\end{example}

Goodrick proves \cite{presburger_bounded_qe} that any parametric Presburger family $\{S_t : t \in \N\}$ is definable by an $\mathcal{L}^+_{\textup{EQP}}$-formula with \emph{polynomially-bounded quantifiers}. Using this result, we now have that every quantified variable $y$ is associated with a condition of the form $0\le y\le f(t)$, where $f:\N\rightarrow\Z$ is a polynomial function of $t$.

\bigskip

\textbf{Step 2:} 
Next, we eliminate all occurrences of the divisibility predicates $D_{f(t)}$. We do this by replacing each variable $x_i$ by an expression $g(t) \cdot u_i + v_i$ where $u_i$ and $v_i$ are new variables, $g$ is a nonnegative common multiple of all the functions $f_1, \ldots, f_n$ which occur in divisibility predicates $D_{f_i}$ in the formula, and with the restriction that $0 \leq v_i < g(t)$. That is, $u_i$ and $v_i$ represent the quotient and remainder when $x_i$ is divided by $g(t)$.

\begin{example} 
\label{Ex:step2}
Let
\[S_t=\{(x_1,x_2)\in\Z^2:\ D_{t+1}(x_1 + 2 x_2)\ \wedge\ D_t(x_1-x_2)\}.\]
Let $g(t) = t(t+1)$ and $x_i = g(t) u_i + v_i$. Then $D_{t+1}(x_1 + 2 x_2)$ is equivalent to $D_{t+1}(v_1 + 2 v_2)$, since $(t+1)\ \big|\ g(t)$. Furthermore, $0\le v_i<t(t+1)$ implies that $v_1+2v_2<3t(t+1)$, and so $D_{t+1}(v_1 + 2 v_2)$ could be replaced by the formula
\[\exists y \left[0 \leq y \le 3t\ \wedge\ y \cdot (t+1) = v_1 + 2 v_2 \right].\]
Replacing both predicates, we define $S'_t$ as the set
\begin{align*}
\Big\{(u_1,u_2,v_1,v_2):\ &\exists y \left[0 \leq y \leq 3t\ \wedge\ y \cdot (t+1) = v_1 + 2 v_2 \right]
\wedge\ \exists z \left[-(t+1) \le z \le t+1\ \wedge\ z \cdot t = v_1 - v_2 \right] \Big\}.
\end{align*}
$S'_t$ is in bijection to $S_t$ under the map $(u_i,v_i)\mapsto x_i = g(t) u_i + v_i$.
\end{example}
Notice that we have eliminated the divisibility predicates, at the expense of introducing some additional polynomially-bounded quantifiers. The result of Step 2 is a new parametric Presburger family $S'_t$ whose points are in bijection with those of $S_t$ and such that $S'_t$ is defined by an $\mathcal{L}_{\textup{EQP}}$-formula with polynomially-bounded quantifiers. Furthermore, the bijection from $S'_t$ to $S_t$ is an \emph{affine reduction}; that is, it is given by an affine linear function whose coordinates are eventually quasi-polynomial in $t$. We will see that this preserves all the properties of $S_t$ that we care about.  

\bigskip

\textbf{Step 3:} We now apply a version of Chen--Li--Sam's ``base $t$'' method \cite{ChLiSam}, in order to finally eliminate the (polynomially-bounded) quantified variables, $y_1,\ldots,y_d$. After reducing to the case that all variables are nonnegative, suppose that we have some $k$ such that, for all $i$, $0\le y_i< t^k$ are bounds on the quantified variables.

We write each $y_i$ ``base $t$,'' using new variables $b_{ij}$:
 \[y_i=\sum_{j=0}^{k-1}b_{ij}t^j,\quad 0\le b_{ij} <t.\]
Suppose that $\ell$ is the largest degree of any polynomial $f(t)$ appearing as a coefficient of some $y_i$. Write each free (unquantified) variable, $x_i$, in base $t$ using new variables $a_{ij}$ and $z_i$:
\[x_i=z_i t^{k+\ell} + \sum_{j=0}^{k+\ell-1}a_{ij}t^j,\quad 0\le a_{ij} < t.\]
Note that we have no bound on the $x_i$, so we must allow an \emph{unbounded} $t^{k+\ell}$ coefficient. The resulting $S'_t$ will be in bijection with the original $S_t$, with the above formulas for $x_i$ and $y_i$ yielding the bijection. As in Step 2, the bijection is an affine reduction. 

We now follow the Chen--Li--Sam method \cite{ChLiSam}, and iteratively look at the $t^0$ coefficients of each term, the $t^1$ coefficients, etc.

\begin{example}
\label{Ex:step3}
Suppose we have the formula\footnote{The reader may notice that this formula is equivalent to a quantifier-free formula in which we replace the variables $y_i$ by their maximum values $t^2-1$. However, such a replacement will not generally work with more complex formulas, whereas the technique of this example can be applied separately to each of the atomic inequalities.}
\[ 0 \leq x_1, x_2 \wedge \exists y_1,y_2 \left[ \left( 0\le y_i<t^2 \right) \wedge  \left( x_1-tx_2 \leq (t+1)y_1+(t+2)y_2 \right)\right].  \]
Replace $y_i$ by $b_{i1}t+b_{i0}$ and $x_i$ by $z_it^3 + a_{i2}t^2+\cdots+a_{i0}$, with $0\le b_{ij}<t$ and with $0\le a_{ij}<t$. That is, $z_1$ and $z_2$ are the only variables not bounded by $t$. The main inequality is now equivalent to 
\begin{align*}&t^4(-z_2)+t^3(z_1-a_{22})+t^2(a_{12}-a_{21}-b_{11}-b_{21})\\
+&t(a_{11}-a_{20}-b_{11}-b_{10}-2b_{21}-b_{20})+(a_{10}-b_{10}-2b_{20}) \leq 0.
\end{align*}

Write this inequality as
\begin{equation} \label{tu-inequality} tu + f_0 \leq 0 \end{equation}
where $f_0 := a_{10}-b_{10}-2b_{20}$. Since the $a$- and $b$-variables are all bounded between 0 and $t-1$, we see that
\[ -3t + 3 \leq f_0 \leq t-1.\]
This implies that we are in one of four exclusive cases:
\[ \left( -3t+3 \leq f_0 \leq -2t \right) \vee \left( -2t+1 \leq f_0 \leq -t \right) \vee \left( -t+1 \leq f_0 \leq 0 \right) \vee \left( 1 \leq f_0 \leq t-1 \right) .\]

Let's look at the case where $-3t+3 \leq f_0 \leq -2t$ 
(we will handle each of the finite number of other cases separately.) Dividing inequality (\ref{tu-inequality}) by $t$ and taking integer parts, we obtain the equivalent inequality
\begin{equation*} \label{ceiling-inequality}  u + \left\lceil \frac{f_0}{t} \right\rceil \leq 0. \end{equation*}
By the hypothesis on $f_0$, $\left\lceil \frac{f_0}{t} \right\rceil = -2$, so we have reduced to the inequality $u-2\le 0$, that is,
\begin{align*}t^3(-z_2)&+t^2(z_1-a_{22})+t(a_{12}-a_{21}-b_{11}-b_{21}) 
+(a_{11}-a_{20}-b_{11}-b_{10}-2b_{21}-b_{20}-2) \leq 0.\end{align*}

Consider the new constant term $f_1 := a_{11}-a_{20}-b_{11}-b_{10}-2b_{21}-b_{20}-2$. Again using the bounds on $a_{ij},b_{ij}$ we have
\[ -6t+4 \leq f_1 \leq t-3.\]
The relevant seven cases are now
\[\left( -6t+4 \leq f_1 \leq -5t \right) \vee \left( -5t+1 \leq f_1 \leq -4t \right) \vee \cdots \vee \left( 1 \leq f_1 \leq t-3 \right) \]
and we can again look at them separately. For example, if $-4t+1 \leq f_1 \leq -3t$, then by again dividing by $t$, the original inequality becomes
\[ t^2(-z_2)+t(z_1-a_{22})+(a_{12}-a_{21}-b_{11}-b_{21}-3) \leq 0. \]
Once again we break into a finite list of cases for the constant term $f_2$. For example, in the case
\[ -2t+1 \leq a_{12}-a_{21}-b_{11}-b_{21}-3 \leq -t, \]
the original inequality becomes equivalent to
\begin{equation} \label{z-inequality} t(-z_2) + (z_1-a_{22}-1) \leq 0.\end{equation}

Since $z_1$ is unbounded, we cannot continue with this method. But notice that the case-defining inequalities $-3t + 3 \leq f_0 \leq -2t$, $-4t+1 \leq f_1 \leq -3t$, and $-2t+1 \leq f_2 \leq 1$ do not involve multiplication by $t$, and the inequality (\ref{z-inequality}) involves no quantified variables.  
\end{example}

This strategy works in general. Our new formula has two types of atomic subformulas: those that involve quantified variables, which are in classical Presburger arithmetic (do not involve multiplication by $t$), and those that use only unquantified variables.

\bigskip

\textbf{Step 4:} We may now apply the classical quantifier elimination procedure to get a logically equivalent quantifier-free formula. Theorem~1.5 from \cite{presburger_bounded_qe} says that this works, because the parameter $t$ does not occur in any of the atomic subformulas involving quantified variables. As discussed before Step 1, we must extend our language to allow $D_c$, divisibility predicates for constants $c$. To summarize, we now have a quantifier-free formula in parametric Presburger arithmetic, with the addition of $D_c$ predicates.

\bigskip

\textbf{Step 5:} Now that we have a quantifier-free formula, we are ready to understand the geometry of these sets $S_t\subseteq\Z^d$. Our formula consists of atomic formulas of the form
\begin{itemize}
\item $\mathbf{f} (t)\cdot \x \le g(t)$ and
\item $D_c\big(\mathbf{f}(t)\cdot\x-g(t)\big)$.
\end{itemize}
If our formula were simply a \emph{conjunction} of such atomic formulas, then we would have that
\[S_t=P_t\cap(\lambda_t+\Lambda_t),\]
where $P_t\subseteq \R^d$ is a polyhedron (which changes as $t$ changes) defined by the linear inequalities, $\Lambda_t\subseteq\Z^d$ a lattice, and $\lambda_t\in\Z^d$ a translation vector, with the latter two defined by the divisibility conditions.

\begin{example}
\label{ex:polyhedron}
If $S_t$ is the set of $x\in\Z$ such that
\[-x\le 0\ \wedge\ tx\le t^2+1\ \wedge\ D_2(x+t)\]
then
\[P_t=[0,(t^2+1)/t]\subseteq\R,\quad \Lambda_t=2\Z,\quad \lambda_t = (t\bmod 2).\]
\end{example}

Of course our $S_t$ may not be simply a conjunction of these atomic formulas, so we first show that $S_t$ can be written as a disjoint union of such sets, using a variant of Disjunctive Normal Form. Then we may concentrate on each piece individually; $\abs{S_t}$, for example, will simply be the sum of the cardinalities of each piece.

So we may assume that $S_t$ is a conjunction of these atomic formulas. Note that a conjunction of formulas of the form $D_c\big(\mathbf{f}(t)\cdot\x-g(t)\big)$ is an ``external'' representation of a translation of a lattice, i.e., defined via constraints, using the language of \cite{Sch}. We'd like an ``internal'' representation, i.e., defined parametrically; in this case, this would be a collection of basis vectors for the lattice, together with a translation vector. 

\begin{example}
\label{Ex:step5a}
In Example \ref{ex:polyhedron}, we converted the external representation $D_2(x+t)$ into the internal representation: lattice $\Lambda_t\subseteq\Z$ with basis $\{2\}$ and translate $\lambda_t=(t\bmod 2)$.
\end{example}

By examining Hermite Normal Forms, we prove that this can be accomplished while preserving EQP properties, and then this basis allows us to apply an affine reduction to eliminate any $D_c$ terms.

\begin{example}
\label{Ex:step5b}
Continuing Example \ref{ex:polyhedron}, our internal representation of $\Lambda_t+ \lambda_t$ implies the following: for odd $t$, any $x$ satisfying $D_2(x+t)$ is of the form $x=2u+1$, for some $u\in \Z$. Applying the affine reduction given by $x=2u+1$ yields (for odd $t$)
\[S'_t=\big\{u\in\Z:\  -(2u+1)\le 0 \wedge t(2u+1)\le t^2+1\big\}.\]
That is, we have eliminated all divisibility terms from the formula.
\end{example}

In other words, we now simply have that $S_t$ is the set of integer points in a parametric polyhedron, defined with linear inequalities of the form $\mathbf{f} (t)\cdot \x \le g(t)$. This is again an ``external'' representation. We show that we can convert to an ``internal'' representation and preserve EQP properties, using row reduction over $\Q(t)$. In this case, an internal representation is a list of vertices (given as ratios of EQPs), of extreme rays of the recession cone (encoding directions, $\vec y$, in which the polyhedron is infinite in the $\vec y$ direction but not in the $-\vec y$ direction), and of a basis for the lineality space (encoding directions $\vec z$ in which it is infinite in both the $\vec z$ and $-\vec z$ directions).

\begin{example}
\label{Ex:step5c}
Continuing Example \ref{ex:polyhedron}, the external representation $-(2u+1)\le 0 \wedge t(2u+1)\le t^2+1$ defines the polyhedron with vertices $0$ and $(t^2-t+1)/2t$.
Since this is a bounded polyhedron, the recession cone and the lineality space are trivial.
\end{example}

\bigskip

\textbf{Step 6:}  
We are now ready to prove Properties 1, 2, and 3. Property 1 will follow from Property 3 trivially. Property 2 follows directly from the Chen, Li, Sam result \cite{ChLiSam}, as long as our polyhedron is bounded. Since we know whether our polyhedron has non-trivial recession cone and lineality space, we know whether it is bounded. If it is unbounded, then $\abs{S_t}$ is either 0 or infinite, and we will show how to figure out which is correct. Property 3 similarly follows from Shen \cite{Shen15b}. We delay discussion of generating functions until Section \ref{sec:pl}, but Property 4 follows from Woods \cite{woods1}.

\subsection{Summary of Outline of Proof.}

Suppose that we are given any parametric Presburger family $S_t \subseteq \Z^d$ defined by an $\mathcal{L}_{\text{EQP}}$-formula $\varphi$. Then we will apply a series of logical equivalences and affine reductions as follows:

\[ \varphi \xleftarrow[\text{(Step 1)}]{\text{logic}} \varphi_1 \xleftarrow[\text{(Step 2)}]{\text{affine}} \varphi_2 \xleftarrow[\text{(Step 3)}]{\text{affine}} \varphi_3 \xleftarrow[\text{(Step 4)}]{\text{logic}} \varphi_4 \xleftarrow[\text{(Step 5)}]{\text{logic + affine}} \varphi_5 \]

where:

\begin{enumerate}
\item $\varphi_1$ is an $\mathcal{L}^+_{\textup{EQP}}$-formula with polynomially-bounded quantifiers,
\item $\varphi_2$ is an $\mathcal{L}_{\textup{EQP}}$-formula with polynomially-bounded quantifiers,
\item $\varphi_3$ is an $\mathcal{L}_{\textup{EQP}}$-formula in which no variable within the scope of a quantifier is multiplied by the parameter $t$,
\item $\varphi_4$ is a \emph{quantifier-free} $\mathcal{L}^+_{\textup{EQP}}$-formula whose only divisibility predicates are $D_c$ for constant functions $c$, and
\item $\varphi_5$ is a disjoint union of conjunctions of atomic $\mathcal{L}_{\textup{EQP}}$-formulas (and hence quantifier-free).
\end{enumerate}

At each stage, the reductions preserve all the properties we are interested in, until we finally reduce to the case of the parametric polyhedra defined by $\varphi_5$. In Step 6, we deal with these polyhedra by applying previously-established combinatorial techniques from Chen-Li-Sam \cite{ChLiSam}, Woods \cite{woods1}, and Shen \cite{Shen15b}.

The sequence of reductions is a little complicated, so one might well wonder whether there is a more direct proof. For instance, one might try to show that any $\mathcal{L}_{\text{EQP}}$-formula is logically equivalent to a quantifier-free formula in a slightly larger language with additional ``well-behaved'' function and relation symbols (the method of \emph{quantifier elimination} from logic), then prove a generalization of the main theorem from \cite{woods1} that quantifier-free parametric Presburger formulas have the EQP behavior we seek. But we already know that quantifier elimination in the original language $\mathcal{L}_{\text{EQP}}$ is impossible (see \cite{presburger_bounded_qe}), and finding a reasonable language for quantifier elimination seems difficult. For example, Kraj\'i\v{c}ek \cite{kraj} gave the name ``two-sorted Presburger arithmetic'' to the complete first-order theory of $\Z$ as an ordered $\Z$-module in a language with two sorts of variables (one for the ring, one for the module). Two-sorted Presburger arithmetic is very similar to the logical system studied in this paper, and Kraj\'i\v{c}ek noted ``it is an interesting open question whether a form of quantifier elimination holds.'' Two-sorted Presburger arithmetic without a relation symbol for $\leq$ was studied by van den Dries and Holly \cite{vdDHolly}, who gave a quantifier elimination theorem for this weaker system in a language including a relation symbol for divisibility and function symbols for $\textup{gcd}$ and a few other closely related arithmetic functions.

\section{Parametric Presburger families and polynomially-bounded quantifiers}
\label{sec:logic}
In this section, we set notation and definitions and review the result on bounding quantifiers in parametric Presburger families proved in \cite{presburger_bounded_qe}. We also define affine reductions and note that they preserve Properties 1 through 3.

\subsection{Parametric Presburger arithmetic}
We will use standard notation and terminology from first-order logic, for which any modern textbook on the subject could serve as a reference (for instance, \cite{enderton} or \cite{hodges}). In particular, a \emph{language} is a set of finitary relational, functional, and constant symbols, usually denoted by $\mathcal{L}$ with decorations. Given a language $\mathcal{L}$, an \emph{$\mathcal{L}$-formula} means a first-order formula in $\mathcal{L}$: we allow basic symbols in $\mathcal{L}$ plus equality, symbols for variables, Boolean operations, and quantifiers.

We let $\mathcal{L}_{Pres}$ be the first-order language with symbols for $0$ and $1$ (constants), $<$ (binary relation for the ordering), $-$ (a unary operation symbol for negation), and $+$ (binary function symbol for addition).

We will always work in the standard model of Presburger arithmetic with universe $\Z$, so whenever we write ``$\models \varphi$'' or ``$\varphi$ holds'' we are referring to the standard interpretation of the language.

$\mathcal{L}^+_{Pres} \supseteq \mathcal{L}_{Pres}$ is the expansion which includes unary predicates $D_c$ for each $c \in \Z$, to be interpreted as divisibility by $c$. By convention, if $c=0$, then ``$D_0(x)$'' is false for every $x \in \Z$, even if $x = 0$.

Now we come to the main definition:

\begin{definition}
A \emph{parametric Presburger formula} is a first-order formula in the language $\mathcal{L}_{\textup{EQP}} := \mathcal{L}_{Pres} \cup \{ \lambda_t\}$. The intended interpretation of $\lambda_t(x)$ is that we are multiplying $x$ by the parameter $t$ which takes some value from $\N$.

Given $t \in \N$ and an $\mathcal{L}_{\textup{\textup{EQP}}}$-formula $\varphi$, we define the $\mathcal{L}_{Pres}$-formula $\varphi_t$ to be the translation of $\varphi$ defined recursively so that each term of the form $\lambda_t(s)$ occurring in $\varphi$ (where $s$ is a term) is replaced by one of the following:

\begin{enumerate}
\item If $t > 0$, then $\lambda_t( s) $ is replaced in $\varphi_t$ by $s + s + \ldots + s$ with $t$ repetitions of $s$;
\item if $t = 0$, then $\lambda_t( s)$ is replaced in $\varphi_t$ by the constant symbol $0$.
\end{enumerate}

\end{definition}

\begin{remark}
Here and below, we will adopt the notational convention of writing ``$f(t) \cdot s$'' for the $\mathcal{L}_{\textup{EQP}}$-term formed from the term $s$ and $f(t) \in \Z[t]$ by repeated applications of the function $\lambda_t$. For instance, the expression $(t^2 + 2) \cdot x_1$ stands for the term $\lambda_t(\lambda_t(x_1)) + x_1 + x_1$.

\end{remark}

When we write $\varphi(\x)$, we mean that all of the free variables occurring in $\varphi$ are listed in the tuple $\x = (x_1, \ldots, x_d)$. (Throughout, boldface letters such as $\A$ and $\x$ will always denote finite tuples.)

The definition of $\varphi_t$ above allows us to talk about the \emph{truth} of $\mathcal{L}_{\textup{EQP}}$-formulas relative to a parameter $t \in \N$: given an $\mathcal{L}_{\textup{EQP}}$-formula $\varphi(x_1, \ldots, x_d)$ whose free variables are contained in $\{x_1, \ldots, x_d\}$ and $(k_1, \ldots, k_d) \in \Z^d$, we will write $$\models \varphi_t(k_1, \ldots, k_d)$$ just in case the $\mathcal{L}_{Pres}$-formula $\varphi_t$ is true with the variable $x_i$ evaluated as $k_i$.

\begin{definition}
A \emph{parametric Presburger family} is a family of sets $\{S_t : t \in \N \}$ such that $$S_t = \{ (k_1, \ldots, k_d) \in \Z^d : \, \, \, \models \varphi_t(k_1, \ldots, k_d) \}$$ for some fixed $d$ and some $\mathcal{L}_{\textup{EQP}}$-formula $\varphi(x_1, \ldots, x_d)$.
\end{definition}

From a logical standpoint, the idea behind parametric Presburger definability is that we are essentially expanding classical Presburger arithmetic by a restricted multiplication function: we allow multiplication of any variable or term by the special parameter variable $t$, but we do not allow multiplication between any of the other variables, and we do not allow quantification over $t$, thus avoiding the complications of sets definable in the full first-order theory of $(\Z; <, +, \cdot)$. 

\begin{definition}
$\mathcal{L}^+_{\textup{EQP}} := \mathcal{L}_{\textup{EQP}} \cup \{D_{f(t)} : f(t) \in \Z[t] \}$, where $D_{f(t)}$ is a unary relation symbol denoting divisibility by the value of $f(t)$. 

Given an $\mathcal{L}^+_{\textup{EQP}}$-formula $\varphi(x_1, \ldots, x_d)$, for any $t \in \N$ and $a_1, \ldots, a_d \in \Z$, we can define the truth value of $\varphi_t(a_1, \ldots, a_d)$ as before, recalling the convention that $D_0(x)$ is always false.
\end{definition}

\begin{definition}
Two $\mathcal{L}^+_{\textup{EQP}}$-formulas $\varphi(x_1, \ldots, x_d)$ and $\psi(x_1, \ldots, x_d)$ are \emph{logically equivalent} just in case for \emph{every} $t \in \N$, the $\mathcal{L}^+_{Pres}$-formulas $\varphi_t$ and $\psi_t$ are logically equivalent; in other words, for every $t \in \N$ and every $(k_1, \ldots, k_d) \in \Z^d$, $$\models \varphi_t(k_1, \ldots, k_d) \Leftrightarrow \models \psi_t(k_1, \ldots, k_d).$$

\end{definition}

It is clear that every $\mathcal{L}^+_{\textup{EQP}}$-formula is logically equivalent to some $\mathcal{L}_{\textup{EQP}}$-formula.

\begin{definition}
\label{def:pb}
  Given an $\mathcal{L}^+_{\textup{EQP}}$-formula $\varphi(x_1, \ldots, x_d, y)$ whose free variables are among $\{x_1, \ldots, x_d, y\}$ and $f \in \Z[t]$, a \emph{polynomially-bounded universal quantifier} applied to $\varphi$ yields $$\forall y \left[ 0 \leq y \leq f(t) \rightarrow \varphi(x_1, \ldots, x_d, y) \right]$$ for some $f \in \Z[t]$. Similarly, a \emph{polynomially-bounded existential quantifier} applied to $\varphi$ yields
\[\exists y\left[0\le y\le f(t)\ \wedge\ \varphi(x_1, \ldots, x_d, y) \right].\]

An \emph{$\mathcal{L}^+_{\textup{EQP}}$-formula with polynomially-bounded quantifiers} is a member of the smallest class of $\mathcal{L}^+_{\textup{EQP}}$-formulas containing all atomic formulas and closed under Boolean combinations and the formation of polynomially-bounded quantifiers.
\end{definition}

We recall the following theorem, which was proved in \cite{presburger_bounded_qe}:

\begin{thm}
\label{qe_main}
Every $\mathcal{L}^+_{\textup{EQP}}$-formula is logically equivalent to an $\mathcal{L}^+_{\textup{EQP}}$-formula with polynomially-bounded quantifiers.
\end{thm}

Also see \cite{presburger_bounded_qe} for a discussion on how this relates to similar previously-known results of Weispfenning and Lasaruk-Sturm \cite{lasaruksturmweakQE}.

\begin{ass}
In any first-order formula $\varphi$, we always assume that no variable occurs in $\varphi$ both as a free variable and as a quantified variable: that is, formulas such as $$( x > 0 ) \wedge (y > 0) \wedge \exists x \left[x + x = y \right]$$ are not allowed (we have to rename the quantified variable and replace the last conjunct by ``$\exists z \left[ z + z = y \right]$'').

\end{ass}

\subsection{Eventual quasi-polynomials and affine reductions}

Recall that EQPs were defined in Definition \ref{def:QP}.

\begin{definition}
\label{EQP_affine}
Given any $d' \times d$ matrix $A(t) = (f_{i,j}(t))$ of EQP functions and any $d'$-tuple $(g_1(t), \ldots, g_{d'}(t))$ of EQP functions, we call the function $F: \Z^d \times \N \rightarrow \Z^{d'}$ given by the rule $$F(x_1, \ldots, x_d, t) = A(t) \cdot (x_1, \ldots, x_d)^T + (g_1(t), \ldots, g_{d'}(t))^T$$ an \emph{EQP-affine function}.
\end{definition}

\begin{definition}
\label{EQP_reduction}
Suppose that $S_t \subseteq \Z^d$ and $S'_t \subseteq \Z^{d'}$ are parametric Presburger families, defined by $\mathcal{L}^+_{\textup{EQP}}$-formulas $\varphi(\x)$ and $\varphi'(\z)$ respectively. We say that $S'_t$ is \emph{affine reducible to} $S_t$ if there is an EQP-affine function $F: \Z^{d'} \times \N \rightarrow \Z^{d}$ such that for every $t \in \Z$, if $F_t$ is the function obtained from $F$ by fixing the value of $t$ for the $(d'+1)$-th coordinate, then $F_t \upharpoonright S'_t$ is a bijection from $S'_t$ onto $S_t$.

\end{definition}

Abusing notation, we will not always distinguish carefully between an $\mathcal{L}^+_{\textup{EQP}}$-definable family and the parametric Presburger set that it defines, and we may say ``$\varphi'$ is affine reducible to $\varphi$'' with the obvious meaning.

The key property of affine reductions for our approach is that they preserve Properties 1, 2, and 3:

\begin{remark}
\label{property_preservation}
If $S_t \subseteq \Z^d$ and $S'_t \subseteq \Z^{d'}$ are families parametrized by $t \in \N$ and $S'_t$ is affine reducible to $S_t$, then if $S'_t$ has any of the Properties 1, 2, or 3, then $S_t$ has the same Properties.
\end{remark}

\begin{proof}
The preservation of Properties 1 and 2 is trivial, since $F$ induces bijections from the sets $S'_t$ onto the $S_t$. For Property 3, if $F: \Z^{d'} \times \N \rightarrow \Z^d$ gives an affine reduction from $S'_t$ to $S_t$ and $\vec x : \N \rightarrow \Z^{d'}$ is an EQP-affine function such that $\vec x(t) \in S'_t$ (whenever $S'_t \neq \emptyset$), then the composition $F \circ \vec x : \N \rightarrow \Z^d$ is an EQP function yielding points in $S_t$ whenever this is nonempty.
\end{proof}

\section{A sequence of reductions}
\label{sec:proof}
In this section we give the details of the six steps of our procedure as outlined in Section \ref{sec:outline}. This will prove Properties 1, 2, and 3 in Theorem \ref{main}. 

\subsection{Step 1}

This is the content of Theorem \ref{qe_main}, proved in \cite{presburger_bounded_qe}.

\subsection{Step 2}
Next we describe an affine reduction which eliminates the divisibility predicates $D_{f(t)}$ from an $\mathcal{L}^{+}_{\textup{EQP}}$ formula, but at the cost of adding additional bounded quantifiers and new variables, cf. Example \ref{Ex:step2}:

\begin{lem}
\label{step_2}
Let $S_t \subseteq \Z^d$ be a parametric Presburger family defined by an $\mathcal{L}^+_{\textup{EQP}}$-formula with polynomially-bounded quantifiers. Then there is a parametric Presburger family $S'_t \subseteq \Z^{d'}$ which is definable by an $\mathcal{L}_{EQP}$-formula with polynomially-bounded quantifiers and such that $S'_t$ is affine reducible to $S_t$.
\end{lem}

\begin{proof}
Fix some $\mathcal{L}^+_{\textup{EQP}}$ formula $\varphi(\x)$ with polynomially-bounded quantification which defines the family $S_t$. Let $f_1, \ldots, f_r$ be all the functions in $\Z[t]$ which occur in divisibility predicates $D_{f_i(t)}$ in $\varphi$. Without loss of generality, none of the $f_i$'s is the zero polynomial (since formulas $D_0(s)$ are trivially false), so $f_i(t)$ is nonzero for sufficiently large $t$. Since $D_f$ and $D_{-f}$ are logically equivalent, we may assume that $f_i(t)$ is eventually positive. Let $g \in \Z[t]$ be the product $f_1f_2\cdots f_r$. Then for sufficiently large $t \in \N$, the value $g(t)$ is a positive common multiple of $f_1(t), \ldots, f_r(t)$; there are only finitely many values of $t$ at which one of $f_1(t), \ldots, f_r(t)$ is nonpositive, and an EQP-affine reduction can handle these cases separately.

Let $z_1, \ldots, z_m$ list all variables occurring in $\varphi$, both free and quantified. Let $u_1, \ldots, u_m$ and $v_1, \ldots, v_m$ be a set of new variables. Let $F: \Z^{2m} \times \N \rightarrow \Z^m$ be the EQP-affine function $$F(u_1, \ldots, u_m, v_1, \ldots, v_m, t) = (g(t) \cdot u_1 + v_1, \ldots, g(t) \cdot u_m + v_m).$$ We will now define a formula $\varphi'$ whose free variables are among $u_1, \ldots, u_m,$ $v_1, \ldots, v_m$ and show that this is affine reducible to $\varphi$ ``via $F$'': this means that the restriction of $F$ to the set defined by $\varphi'$, ignoring any $u_i$ or $v_i$ which does not occur free, yields a bijection onto the set $S_t$ defined by $\varphi$.

In fact, we will recursively construct $\mathcal{L}_{\textup{EQP}}$-formulas $\psi'$ for every subformula $\psi$ of $\varphi$, starting with the atomic subformulas, such that every $\psi'$ has polynomially-bounded quantification and the parametric Presburger family defined by $\psi$ is affine reducible to that defined by $\psi'$ via the map $F$, and the family $S'_t$ that we seek will be that defined by $\varphi'$.

\medskip

The formula $\psi'$ will be defined as $\theta \wedge \psi''$ where $\theta$ is the conjunction $$\bigwedge_{i = 1}^m \left[ 0 \leq v_i < g(t) \right]$$ and the formula $\psi''$ is defined by recursion. As we define the formula $\psi''$, we will also show inductively that the following holds:

\medskip

$(*)$ For any sequence of values $(\A, \B) = (a_1, \ldots, a_m, b_1, \ldots, b_m)$ which satisfies $\theta$, we have $\models (\psi'')_t(\A, \B)$ if and only if $\models \psi_t(F(\A, \B, t))$.

\medskip

Note that $(*)$ immediately implies that $\psi'$ is affine reducible to $\psi$ via $F$.

\medskip

$\bullet$ For any term $s$ in the variables $z_1, \ldots, z_m$, let $s( g \cdot \U + \V / \z)$ be the result of replacing each instance of the variable $z_i$ by $g \cdot u_i + v_i$. Then if $\psi$ is an atomic formula of the form $s_1 = s_2$, we let $\psi''$ be the atomic formula $$s_1( g \cdot \U + \V / \z) = s_2( g \cdot \U + \V / \z).$$ If $\psi$ is an atomic formula of the form $s_1 \leq s_2$, then $\psi''$ is $$s_1( g \cdot \U + \V / \z) \leq s_2( g \cdot \U + \V / \z).$$ It is clear that in this case that $(*)$ holds.

\medskip

$\bullet$ For any atomic formula $\psi$ of the form $D_{f_j(t)}(s)$ (where $s$ is a term), let $h_s(t)\in\Z[t]$ be such that $\abs{s^t(a_1, \ldots, a_m)} \leq h_s(t)$, for all $t\in\N$ and all $a_1,\ldots,a_m$ with $0\le a_i<g(t)$ (since, for fixed $t$, $s$ is linear in $a_1, \ldots, a_m$, we only need check the bound on the finite set of $2^d$ EQPs obtained by setting each $a_i$ to $0$ or to $g(t)$, in the term $s(a_1, \ldots, a_m)$). Let $s(\V / \z)$ be the result of replacing each instance of each variable $z_i$ by $v_i$. Now let $\psi''$ be the formula $$\exists y \left[0\leq y \leq h_s(t) \ \wedge \ \big(f_j(t) \cdot y = s(\V / \z) \ \vee\ f_j(t)\cdot (-y)= s(\V / \z)\big)\right],$$ where $y$ is a new variable. Note that $\psi''$ has polynomially-bounded quantification. The hypothesis $(*)$ asserts that checking that the term $s$ applied to the $g(t) \cdot u_i + v_i$ is divisible by $f_j(t)$ is equivalent to checking the divisibility of $s$ applied to the $v_i$: this is true because $g(t)$ is divisible by $f_j(t)$.

\medskip

$\bullet$ $(\psi_1 \wedge \psi_2)'' = \psi_1'' \wedge \psi_2''$, $(\neg \psi)'' = \neg (\psi'')$, $(\psi_1 \vee \psi_2)'' = \psi_1'' \vee \psi_2''$, and $(\psi_1 \rightarrow \psi_2)'' = \psi_1'' \rightarrow \psi_2''$. It is routine to check that if $(*)$ holds for the formulas $\psi_i$ and $\psi''_i$, then they continue to hold for the Boolean combinations above. Furthermore, if the formulas $\psi_i''$ have bounded quantification, then so do all of these Boolean combinations.

\medskip

$\bullet$ If $\psi$ is a formula of the form $\exists z_i \left[ 0 \leq z_i \leq h(t) \wedge \theta \right]$ (with polynomially-bounded quantification on the outside), then we define $\psi''$ as $$\exists u_i [ 0 \leq u_i \leq h(t) \ \wedge$$ $$ \exists v_i \left[\left(0 \leq v_i < g(t) \right) \wedge \left(0 \leq g(t) \cdot u_i + v_i \leq h(t)\right) \wedge \theta''(g(t) \cdot u_i + v_i / z_i) \right] ],$$ where ``$\theta''(g \cdot u_i + v_i / z_i)$'' means the result of substituting each instance of the free variable $z_i$ in the formula $\theta$ with the term $g \cdot u_i + v_i$. Again, the new formula $\psi''$ has polynomially-bounded quantification under the inductive hypothesis that $\theta''$ does, and $(*)$ is also immediate by induction.
\end{proof}

\begin{remark}
We note in passing that the same argument can be applied to $\mathcal{L}_R$-definable families (in the language of \cite{presburger_bounded_qe}) for any ring $R$ of $\Z$-valued functions, with only minor modifications. This gives us the following: for any $R$-parametric Presburger family $S_t \subseteq \Z^d$, there is an $R$-parametric Presburger family $S'_t \subseteq \Z^{d'}$ which is definable by an $\mathcal{L}_R$-formula with $R$-bounded quantifiers and such that $S'_t$ is $R$-affine reducible to $S_t$.
\end{remark}

\subsection{Step 3}
This next step involves writing the variables of our formula in ``base $t$,'' cf. Example \ref{Ex:step3}. To do this, it is most convenient to show that we can reduce to the case where the free variables are nonnegative (by Definition \ref{def:pb} for polynomially-bounded quantifiers, the quantified variables are already nonnegative). Indeed, if our free variables are among $(x_1,\ldots,x_d)$, we partition $\Z^d$ into $2^d$ pairwise disjoint regions $X_1, \ldots, X_{2^d}$ according whether each free variable $x_i$ is negative or nonnegative. Now let $S^k_t := X_k \cap S_t$ for each $k \in \{1, \ldots, 2^d\}$, and note that $S^k_t$ is a parametric Presburger family which is EQP-affine reducible to a parametric Presburger family living in $\N^d$ (by inverting the values of negative coordinates). Then $S_t$ is the disjoint union of these $S^k_t$, and the following remark shows that we may examine each $S^k_t$ separately.

\begin{remark}
\label{disj_union}
If we can write $S_t$ as a union $S^1_t \cup \cdots \cup S^k_t$ of pairwise disjoint parametric Presburger families, then to prove that $S_t$ satisfies Properties 1, 2, and 3 (as in Theorem~\ref{main}), it is sufficient to prove these for each $S^i_t$ separately. In particular, Property 1 follows because the union of eventually periodic sets is eventually periodic, Property 2 follows because the sum of EQP's is an EQP, and Property 3 follows because we may specify an $\x(t)$ from any of the (nonempty) sets in the disjoint union.
\end{remark}

The following is the main Proposition encompassing Step 3:

\begin{prop} \label{quantifier_separation}
Let $S_t \subseteq \N^d$ be a parametric Presburger family which is definable by an $\mathcal{L}_{\textup{EQP}}$-formula with polynomially-bounded variables $y_1, \ldots, y_m$. Then there is an EQP-parametric Presburger family $S'_t$ that is EQP-affine reducible to $S_t$ and such that $S'_t$ is definable by a $\mathcal{L}_{\textup{EQP}}$-formula $\varphi(\A, \z)$ whose free variables are among the variables $\A$ and $\z$ and whose quantified variables are listed in the finite tuple $\B$ (which is assumed to be disjoint from $\A$ and $\z$), and satisfying the following conditions:

\begin{enumerate}
\item The atomic subformulas of $\varphi$ are all of one of two types: 
\begin{enumerate}

\item Type 1, which are inequalities involving only variables in $\A$, $\B$, and $t$, and which do not involve any multiplication by $t$;
\item Type 2, which are inequalities involving only $t$ and the free variables (from $\A$ or from $\z$);
\end{enumerate}

\item Any free variable from $\A$ is explicitly bounded between $0$ and $t-1$ (inclusive) by an atomic subformula of $\varphi$ which is outside the scope of any quantifier or disjunction; and

\item Any quantified variable from $\B$ is bounded between $0$ and $t-1$ (inclusive).

\end{enumerate}
\end{prop} 

In order to prove Proposition~\ref{quantifier_separation}, we adapt the approach of Chen, Li, and Sam in~\cite{ChLiSam} of writing each variable ``base $t$.'' We cannot directly apply their results because of the interaction between the potential unboundedness of the free variables and the existence of (bounded) quantified variables. Specifically, their key Lemma 3.2 uses sets defined only by equalities and nonnegativity constraints. We need to work with general linear inequalities and cannot apply the usual procedure of introducing slack variables to convert them into equalities because of the presence of quantifiers.

For example, consider the set $S$ given by 
$$ S = \{ x \in \Z: x \geq 0 \wedge \exists y \left[ (y \geq 0) \wedge ( x + y \leq 2 )\right] \} \; = \; \{ 0,1,2 \} .$$
Since $x$ and $y$ are nonnegative, we could introduce a slack variable $w$ to obtain the set $S' \subset \Z^2$ given by
$$ S' = \{(x,w): x \geq 0 \wedge w \geq 0 \wedge \exists y \left[(y \geq 0) \wedge ( x+y+w = 2) \right] \} $$
$$ = \; (0,0), (0,1), (0,2), (1,0), (1,1), (2,0) \},$$

but obviously $S$ and $S'$ are not in bijection. 

\begin{lem} \label{base_t}
  For every parametric Presburger family $S'_t$ defined by an $\mathcal{L}_{\textup{EQP}}$-formula $\varphi'$ with polynomially-bounded quantifiers, there is an $M \geq 0$ and a parametric Presburger family $S_t''$ which is affine reducible to $S_t$ and is definable by an $\mathcal{L}_{\textup{EQP}}$-formula $\varphi''$ that has the following properties:
  \begin{enumerate}
  \item The quantified variables $\{b_{ij}\}$ in $\varphi''$ all satisfy
    $0 \leq b_{ij} \leq t-1$.
  \item The free variables in $\varphi''$ are of two types: $\{a_{ij}\}$ with $0 \leq a_{ij} \leq t-1$ and $\{z_i\}$ with no a priori bound.
  \item $\varphi''$ has no equalities and every inequality (other than the bounds on the variables) takes the form
    \begin{equation} \label{f_inequality} \sum_{s \geq 0} f_s t^s \leq 0 \end{equation}
 where no $f_s$ contains $t$, $f_s$ does not contain any unbounded free variables if $s < M$, and $f_s$ does not contain any quantified variables if $s \geq M$.
    \end{enumerate}  
\end{lem}

\begin{proof}
 Let $\{x_1, \dots, x_d\}$ be the free variables and $\{y_1,\dots,y_m\}$ the quantified variables in $\varphi'$.  Let $k \in \N$ be such that $0 \leq y_i \leq t^k -1$ for every $i$ and sufficiently large $t$ (such a $k$ is obtainable, because each $y_i$ is polynomially-bounded; the finite set of $t$ for which the bound doesn't hold may be handled individually). Also let $\ell$ be the highest power of $t$ that is multiplied by any $y_i$ in $\varphi'$. Define a $\Z[t]$-affine (hence \textup{EQP}-affine) function $G: \Z^{d \times (k+\ell+1)} \to \Z^d$ by
 $$ \begin{bmatrix} x_1 \\ \vdots \\ x_d \end{bmatrix} = 
  G \left( \begin{bmatrix}  a_{1,0} & \dots & a_{1,k+\ell-1} & z_1  \\
    \vdots &    & \vdots & \vdots \\
     a_{d,0} & \dots & a_{d,k+\ell-1} & z_d
  \end{bmatrix} \right) =
  \begin{bmatrix} \sum_{j=0}^{k + \ell - 1} t^j a_{1,j} +  t^{k+\ell} z_1 \\
    \vdots \\ \sum_{j=0}^{k + \ell - 1} a_{d,j}t^j + t^{k+\ell} z_d \end{bmatrix}  $$

  That is, we write each $x_i$ base $t$ (in the sense of \cite{ChLiSam}) up to the power $t^{k+\ell-1}$ and write the remaining part of $x_i$ as $t^{k + \ell}z_i$. So there is a unique inverse image of any point $(x_1, \dots x_d)$ that satisfies the inequalities $0 \leq a_{i,j} \leq t-1, 0 \leq z_i$ for $i=1, \dots, d$. We define $S_t''$ to be the set of these preimages, so that $G$ is an affine reduction of $S_t''$ to $S_t'$.
  
  Next, in the formula defining $S_t''$, replace each $y_i$ by $\sum_{j=0}^{k-1} b_{i,j} t^j$ and the bounded quantifier $\exists y_i: \; 0 \leq y_i \leq t^k - 1$ by
  the series of bounded quantifiers $\exists b_{i,0}, \dots, b_{i,k-1}: \; 0 \leq b_{i,0}, \dots, b_{i,k-1} \leq t-1$. Let $\varphi''$ be the resulting formula. 

  Since each unbounded free variable $z_i$ is multiplied by $t^{k + \ell}$ each time it is introduced, the condition on terms of $t$-degree lower than $k + \ell$ is immediate. Also each bounded variable $b_{i,j}$ appears by replacing $y_i$ by $b_{i,j} t^j$; thus the powers of $t$ that are multiplied by $b_{i,j}$ in $\varphi''$ are $j$ more than the powers of $t$ that are multiplied by $y_i$ in $\varphi'$. By the choice of $k$ and $\ell$, if we let $M = k + \ell$ then the final condition is also satisfied. 
\end{proof}

\begin{lem} \label{degree_reduction}
  Suppose $\sum_s f_s t^s \leq 0$ is an atomic subformula of an $\mathcal{L}_{\textup{EQP}}$-formula $\varphi''$ as in the conclusion of Lemma~\ref{base_t}, and further assume that $f_0$ does not contain any of the unbounded free variables $z_1, \dots, z_d$. Then for $t \gg 0$, $\sum_s f_s t^s \leq 0$ is logically equivalent to a Boolean combination of:
  \begin{itemize}
  \item inequalities in Presburger arithmetic (allowing $t$ as an ordinary Presburger variable), and
  \item parametric Presburger inequalities of the form $\sum_s f'_s t^s \leq 0$ such that $f_s' = f_{s+1}$ for $s \geq 1$ and $f_0' = f_1 + h$ where $h$ is constant. 
  \end{itemize}
 \end{lem}

\begin{proof}
  Since $f_0$ does not contain any of the variables $z_1, \dots z_d$, write
  \begin{equation} f_0 = c_0 + \sum_{i=1}^d \sum_{j=o}^{k+\ell-1} c_{i,j}a_{i,j}  + \sum_{i=1}^m \sum_{j=0}^{k-1} c'_{i,j}b_{i,j} \end{equation}
  where all of the coefficients $c_0$, $c_{ij}$ and $c'_{ij}$ are in $\Z$. Let $c$ be the maximum among the absolute values of all of these coefficients and let $Q$ be the total number of terms in $f_0$. Since each variable in $f_0$ is bounded between 0 and $t-1$, we have for $t \gg 0$ that 
  \begin{equation} \label{first_Q_inequality} -Qc(t-1) \leq f_0 \leq Qc(t-1). \end{equation}
 
  We can loosen~(\ref{first_Q_inequality}) to obtain
  \begin{equation} \label{second_Q_inequality} -Qct < f_0 \leq Qct  \end{equation}
  which is equivalent to the finite disjunction
  \begin{equation} \label{Q_cases} \bigvee_{h=-Qc+1}^{Qc} (h-1)t < f_0 \leq ht. \end{equation}

  Now fix a particular value of $h$. With $(h-1)t < f_0 \leq ht$, the inequality $\sum_{s \geq 0} f_s t^s \leq 0$ is equivalent to $\sum_{s \geq 1} f_st^s \leq -ht$. Adding $ht$ to both sides and dividing by $t$, we obtain $\sum_{s \geq 0}f'_st_s \leq 0$, where indeed $f'_s = f_{s+1}$ for $s \geq 1$ and $f'_0 = f_1 + h$.  
\end{proof}

\begin{proof}[Proof of Proposition \ref{quantifier_separation}]
  Given a parametric Presburger family $S_t$, consider the affine equivalent family $S_t''$ and its defining formula $\varphi''$ that are obtained from Lemma~\ref{base_t}. Let $\sum_{s \geq 0} f_s t^s \leq 0$ be one of the inequalities occurring in $\varphi''$, and $k$ and $\ell$ be defined as in the proof of Lemma~\ref{base_t}. If $k=0$, then there are no quantified variables and we do not need to modify the inequality.
  
  If $k \geq 1$, then in particular $k + \ell \geq 1$, so $f_0$ does not contain any of the unbounded free variables. Thus we can apply Lemma~\ref{degree_reduction}. Consider one of the resulting inequalities $\sum_{s \geq 0} f'_s t^s \leq 0$ that is not in Presburger arithmetic. If $k + \ell = 1$, then by Lemma~\ref{base_t}, $f_s$ does not contain any quantified variables for $s \geq 1$, and then by Lemma~\ref{degree_reduction}, $f_s'$ does not contain quantified variables for any value of $s$ whatsoever. That is, the inequality $\sum_{s \geq 0} f'_s t^s \leq 0$ is free of quantified variables. On the other hand, if $k + \ell > 1$, then $f_0' = f_1 + h$ does not contain any unbounded free variables and so we can apply Lemma~\ref{base_t} again. In fact we can repeat this process $k+\ell$ times. For the same reason as in the case of $k+\ell=1$, the resulting inequalities are of two types: those that contain only Presburger arithmetic and those that are free of quantified variables.
\end{proof}

\subsection{Step 4}

Now we apply a quantifier elimination algorithm for ``classical'' Presburger arithmetic to the output of Step 3. The algorithm we will use is that described in \cite{cooper}.

Say that an $\mathcal{L}^+_{\textup{EQP}}$-formula \emph{has no division by $t$} if it involves only divisibility relations $D_c$ where $c$ is a constant (with no dependence on $t$).

A term or formula in $\mathcal{L}^+_{\textup{EQP}}$ \emph{has no multiplication by $t$} if it contains no terms involving multiplication by any $f(t) \in \textup{EQP}$ of degree $\geq 1$.

\begin{prop}
\label{step4}
Suppose that $\varphi(\A, \z)$ is an $\mathcal{L}^+_{\textup{EQP}}$-formula with no division by $t$ whose free variables are among the variables $\A$ and $\z$ and whose quantified variables are listed in the finite tuple $\B$ (which is assumed to be disjoint from $\A$ and $\z$). We further assume:

\begin{enumerate}
\item The atomic subformulas of $\varphi$ are all of one of three types: 
\begin{enumerate}

\item Type 1, which are inequalities involving only variables in $\A$, $\B$, and $t$, and which do not involve any multiplication by $t$;
\item Type 2, which are inequalities involving only $t$ and the free variables (from $\A$ or from $\z$); and
\item Type 3, which are divisibility conditions $D_c(s)$ where $c$ is a constant and $s$ is a term which does not involve any multiplication by $t$;
\end{enumerate}

\item Any free variable from $\A$ is explicitly bounded between $0$ and $t-1$ (inclusive) by an atomic subformula of $\varphi$ which is outside the scope of any quantifier or disjunction; and

\item Any quantified variable from $\B$ is bounded between $0$ and $t-1$ (inclusive).

\end{enumerate}

\textbf{THEN} $\varphi$ is logically equivalent to a quantifier-free formula in $\mathcal{L}^+_{\textup{EQP}}$ which has no division by $t$.
\end{prop}

\begin{proof}
This follows from Theorem~1.5 of \cite{presburger_bounded_qe}.
\end{proof}

\begin{example}
A typically complex formula satisfying the hypotheses of Proposition~\ref{step4} is $$\varphi(a_0, a_1) = (0 \leq a_0 < t) \wedge (0 \leq a_1 < t)$$ $$\wedge \exists b_0 \exists b_1 \left[(0 \leq b_0, b_1 < t) \wedge (t < 2a_0 + 3b_0 + 7b_1 \leq 2t) \wedge (2t < a_1 + 2b_1 \leq 3t) \right].$$ This contains no atomic subformulas of Type 2 involving $z_i$ variables, but adding such a subformula would not make the quantifier elimination process any more difficult. Also, the bounds on $a_0$ and $a_1$ outside the quantifiers will have no effect on the procedure.

We work outwards starting from the inner quantifier $\exists b_1$. Using the notation of the procedure, the least common multiple $M$ of the $b_1$ coefficients is $14$, so we first multiply substitute $b_1'$ for $14 b_1$ and get that the subformula of $\varphi$ bounded by the scope of $\exists b_1 \ldots$ is equivalent to $$\exists b'_1 \, \, (0 \leq b_0 < t) \wedge (0 \leq b'_1 < 14t)$$ $$\wedge (2t < 4a_0 + 6 b_0 + b'_1 \leq 4t) \wedge (14t < 7 a_1 + b'_1 \leq 21t) \wedge D_{14}(b'_1).$$ Now as in the proof, within the quantifier $\exists b'_1$ we have atomic subformulas of the three types:

\bigskip

(A$'$): $b'_1 < 14t$, $b'_1 < 4t - 4 a_0 - 6 b_0 + 1$, and $b'_1 < 21t - 7a_1 + 1$;

(B$'$): $-1 < b'_1$, $2t - 4a_0 - 6b_0 < b'_1$, and $14t - 7a_1 < b'_1$; and

(C$'$): $D_{14} (b'_1)$.

\bigskip

The procedure replaces the entire quantified subformula bounded by the scope of $\exists b'_1$ by a disjunction $\psi'$ of $42$ quantifier-free formulas, one for each choice of a constant $j \in \{1, \ldots, 14\}$ and one of the three expressions B$'$ listed above. For example, fixing the second subformula of type B$'$ yields the following disjunction, which forms a part of $\psi'$: $$\bigvee_{j=1}^{14} (0 \leq b_0 < t) \wedge (0 \leq 2t - 4a_0 - 6b_0 + j \leq 14t) \wedge (2t < 4 a_0  + 2t - 4a_0  + j \leq 4t)$$ $$\wedge(14t < 7a_1 + 2t - 4a_0 - 6b_0 +j \leq 21t ) \wedge D_{14}(2t - 4a_0 - 6b_0 + j) .$$ (Note that for any fixed values of $t, a_0,$ and $b_0$, exactly one of the $14$ values of $j$ will make the last conjunct with $D_{14}$ true, so in practice one does not need to evaluate all of the inequalities for every possible value of $j$.)

Now if we were to continue the quantifier elimination procedure for $\varphi$ (which we will not), the next step would be to eliminate the quantifier $\exists b_0$. We only note that this time we have a subformula of Type 3, $D_{14}(2t - 4a_0 - 6b_0 +j)$, in the input formula, even though the original input formula $\varphi$ contained no such divisibility predicates. The important thing is that it has the correct form for a Type 3 subformula: we are dividing only by the constant $14$ and there is no multiplication of $t$ with another variable.

\end{example}

\subsection{Step 5}

\begin{lem}
Let $S_t\subseteq \Z^d$ be defined by an $\mathcal{L}^+_{\textup{EQP}}$ formula, $\varphi(\x)$, which is quantifier-free and has no division by $t$ (the only divisibility relations are of the form $D_c$, where $c$ is a constant). Then $S_t$ can be written as a finite disjoint union,
\[S_t=\bigsqcup_{i=1}^n S_{i,t},\]
where $S_{i,t}$ is defined by an $\mathcal{L}^+_{\textup{EQP}}$ formula, $\varphi_i(\x)$, that is a conjunction of atomic formulas, each of the form:
\begin{itemize}
\item $\mathbf{f} (t)\cdot \x \le g(t)$ or
\item $D_c\big(\mathbf{f}(t)\cdot\x-g(t)\big)$,
\end{itemize}
where $\vec f:\N\rightarrow \Z^d$ has $\textup{EQP}$ coordinates, $g:\N\rightarrow \Z$ is in $\textup{EQP}$, and $c\in\Z$ is constant.
\end{lem}

\begin{proof}
Using that 
\[\neg(\mathbf{f}(t)\cdot \x \le g(t))\quad\text{logically equivalent to}\quad -\mathbf{f}(t)\cdot \x\le -g(t)-1\]
and that 
\[\neg\big(D_c(\mathbf{f}(t)\cdot\x-g(t))\big)\quad\text{logically equivalent to}\quad \bigvee_{k=1}^{c-1} D_c\big(\mathbf{f}(t)\cdot\x-g(t)+k\big),\]
we may assume that every atomic formula  in $\varphi(\x)$ is of the form $\mathbf{f}(t)\cdot \x \le g(t)$ or $D_c\big(\mathbf{f}(t)\cdot\x-g(t)\big)$.

Enumerate the atomic formulas in $\varphi(\x)$ by $r_1, \ldots,r_m$. For $1\le j\le m$, define $R_j$, a set of atomic formulas, as follows:
\begin{itemize}
\item If $r_j$ is of the form $\mathbf f(t)\cdot \mathbf{x} \le g(t)$, then
\[R_j=\{\mathbf f(t)\cdot \mathbf{x} < g(t),\ \mathbf f(t)\cdot \mathbf{x} > g(t),\ \mathbf f(t)\cdot \mathbf{x} = g(t)\}\]
\item If $r_j$ is of the form $D_c\big(\mathbf f(t)\cdot\mathbf{x}-g(t)\big)$, then
\[R_j=\{D_c\big(\mathbf f(t)\cdot\mathbf{x}-g(t)+k\big):\ 0\le k\le c-1\}.\]
\end{itemize}

For a given $j$, every element of $\Z^d$ satisfies exactly one $s_j\in R_j$. Also, for a given $s_j\in R_j$, either every element of $\Z^d$ satisfying $s_j$ also satisfies $r_j$, or every element of $\Z^d$ satisfying $s_j$ fails to satisfy $r_j$. For any choice $(s_1,\ldots,s_m)\in R_1\times\cdots\times R_m$, define $S_{(s_1,\ldots,s_m),t}\subseteq\Z^d$ with the formula $s_1\wedge\cdots\wedge s_m$. These $S_{(s_1,\ldots,s_m),t}$ partition $\Z^d$. Furthermore, for a given $S_{(s_1,\ldots,s_m),t}$ either every $\mathbf x\in S_{(s_1,\ldots,s_m),t}$ satisfies $\varphi(\mathbf x)$ or every $\mathbf x\in S_{(s_1,\ldots,s_m),t}$ fails to satisfy $\varphi(\mathbf x)$. Therefore $S_t$ is the disjoint union of those $S_{(s_1,\ldots,s_m),t}$ that do satisfy $\varphi(\mathbf x)$. Using that $<$, $>$, and $=$ may be rewritten with $\le$, these sets are of the required form.
\end{proof}

Since $S_t$ has now been written as a disjoint union, we may examine each piece separately; see Remark~\ref{disj_union} above.

Our next task is to eliminate the $D_c(\cdot)$ atomic formulas with an affine reduction. To do this, let's look at a conjunction of formulas, each of the form $D_c\big(\mathbf f(t)\cdot\mathbf{x}-g(t)\big)$. For each $t$, this will be a translate of a lattice. Using the language of \cite{Sch}, this could be called an ``external'' representation of this lattice translate, i.e., defined via constraints. We'd like an ``internal'' representation, i.e., defined parametrically; in this case, this would be a collection of basis vectors for the lattice, together with a translation vector. The content of the following proposition is that we can achieve this in an \textup{EQP} way, cf. Example \ref{Ex:step5a}. Note that the proposition even applies when the divisibility relations are of the form $D_{h(t)}$, where $h(t)\in \textup{EQP}$.

\begin{prop}
\label{prop:Hermite}
Let $S_t\subseteq\Z^d$ be defined by
\[\bigwedge_{j=1}^n D_{h_j(t)}\big(\mathbf f_j(t)\cdot\mathbf x -g_j(t)\big),\]
where $\mathbf f_j:\N\rightarrow \Z^d$ has $\textup{EQP}$ coordinates, $g_j,h_j:\N\rightarrow \Z$ in $\textup{EQP}$, and 
$h_j(t)$ is eventually nonzero. Then there exist $N$ and a period $m$, such that, for fixed $i$ with $0\le i\le m-1$, either
\begin{itemize}
\item $S_t$ is empty for all $t\in\N$ with $t\ge N$ and $t\equiv i\bmod m$, or 
\item there exist $\mathbf u_0,\mathbf u_1,\ldots,\mathbf u_r:\N\rightarrow\Z^d$ whose coordinate functions are polynomials in $\Z[t]$, such that, for all $t\in\N$ with $t\ge N$ and $t\equiv i\bmod m$, we have that
\[S_t=\mathbf u_0(t)+\Z \mathbf u_1(t)+\cdots+\Z \mathbf u_r(t),\]
where $\mathbf u_1(t),\ldots,\mathbf u_r(t)$ are linearly independent.
\end{itemize}
\end{prop}

\begin{proof}
Note that $S_t$ is the set of $\mathbf x\in\Z^d$ such that there exist $\mathbf y\in\Z^n$, such that
\[\mathbf f_j(t)\cdot \mathbf x+h_j(t)y_j=g_j(t),\]
for all $1\le j\le n$. We create an $n\times (d+n)$ matrix $A$, with entries in $\textup{EQP}$, as follows: let $A=[F\ D]$, where $F$ is the $n\times d$ matrix whose $j^\text{th}$ row is the coordinates of $\mathbf f_j(t)$, and $D$ is the diagonal matrix whose $(j,j)$ entry is $h_j(t)$. Let $\mathbf g$ be the $n\times 1$ column vector whose $j^{th}$ entry is $g_j(t)$. Regarding $\mathbf x$ and $\mathbf y$ as column vectors, we have that $S_t$ is the set of $\mathbf x\in\Z^d$ such that there exists $\mathbf y\in\Z^n$, such that
\[A\left[\begin{matrix} \mathbf x\\ \mathbf y\end{matrix}\right] = \mathbf g.\]

Using Corollary 2.8 of \cite{CW}, $A$ may be put into \emph{Hermite normal form}. That is, letting $r$ be the rank of $A$ (we actually have $r=n$ here, because the matrix $D$ ensures that $A$ is full row rank), there exist an $n\times r$ matrix $B$ and a $(d+n)\times(d+n)$ matrix $U$ such that
\begin{itemize}
\item $AU=[B\ 0]$,
\item all entries above the main diagonal of $B$ are 0,
\item all entries along the main diagonal of $B$ are positive,
\item all entries below the main diagonal of $B$ are nonnegative, and each row achieves its maximum uniquely on the main diagonal,
\item $U$ is unimodular, i.e., determinant $\pm 1$ ($U$ encodes the elementary column operations transforming $A$ into $B$), and
\item all entries are in $\textup{EQP}$.
\end{itemize}

Corollary 5.3b of \cite{Sch} shows that the set of integer solutions to $A\left[\begin{matrix} \mathbf x\\ \mathbf y\end{matrix}\right] = \mathbf g$ is nonempty if and only if $B^{-1}\mathbf g$ is an integer. The entries of $B^{-1}$ may be ratios of $\textup{EQP}$s; the set of $t$ for which $B^{-1}\mathbf g$ is integral will be eventually periodic. Therefore we may assume we are looking at a residue class $i\bmod m$ such that $B^{-1}\mathbf g$ is integral. Then Corollary 5.3c of \cite{Sch} shows that the set of integer solutions to $A\left[\begin{matrix} \mathbf x\\ \mathbf y\end{matrix}\right] = \mathbf g$ can be written in the form
\[\mathbf v_0(t)+\Z \mathbf v_1(t)+\cdots+\Z \mathbf v_d(t),\]
where $\mathbf v_1(t),\ldots,\mathbf v_d(t)$ are linearly independent; in particular,
\[\mathbf v_0=U\left[\begin{matrix}B^{-1}\mathbf g\\ 0\end{matrix}\right]\qquad\text{and}\qquad \mathbf v_j=(n+j)^{\text{th}}\text{ column of $U$,} \]
for $1\le j\le d$, and so the entries of $\mathbf v_j:\N\rightarrow\Z^{d+n}$ are in $\textup{EQP}$. For $0\le j\le d$, let $\mathbf w_j$ be the first $d$ coordinates of $\mathbf v_j$. Then 
\[S_t=\mathbf w_0(t)+\Z \mathbf w_1(t)+\cdots+\Z \mathbf w_d(t).\]
We are almost finished; the only problem is that the $\mathbf w_1,\ldots,\mathbf w_d$ may not be linearly independent. Let $A'$ be the $d\times d$ matrix whose columns are $\mathbf w_1,\ldots,\mathbf w_d$, and let $r$ be the rank of $A'$. Putting $A'$ in Hermite normal form, we have a $d\times d$ unimodular matrix $U'$ and a $d\times r$ matrix $B'$, such that $A'U'=[B'\ 0]$ with restrictions on the entries of $B'$ as above. Let $\mathbf u_1\ldots,\mathbf u_r$ be the columns of $B'$, which are linearly independent. Since elementary column operations don't affect the $\Z$-span of the columns,
\[\Z \mathbf w_1(t)+\cdots+\Z \mathbf w_d(t)=\Z \mathbf u_1(t)+\cdots+\Z \mathbf u_r(t).\]
Taking $\mathbf u_0=\mathbf w_0$, we have the desired property:
\[S_t=\mathbf u_0(t)+\Z \mathbf u_1(t)+\cdots+\Z \mathbf u_r(t),\]
where $\mathbf u_1(t),\ldots,\mathbf u_r(t)$ are linearly independent.
\end{proof}

To recall, we have written the set we are interested in as a disjoint union of sets, each defined by a conjunction of atomic formulas of the form $\mathbf f(t)\cdot \mathbf{x} \le g(t)$ and $D_c\big(\mathbf f(t)\cdot\mathbf{x}-g(t)\big)$. We now show that we can eliminate the $D_c(\cdot)$ formulas, with an affine reduction, cf. Example \ref{Ex:step5b}. In fact, this would work even for divisibility by non-constant $h(t)$:

\begin{lem}
Let $S_t\subseteq \Z^d$ be defined by an $\mathcal{L}^+_{\textup{EQP}}$ formula, $\varphi(\vec{x})$, that is a conjunction of atomic formulas, each of the form
\begin{itemize}
\item $\vec f(t)\cdot \vec{x} \le g(t)$ or
\item $D_{h(t)}\big(\vec f(t)\cdot\vec{x}-g(t)\big)$,
\end{itemize}
where $\vec f:\N\rightarrow \Z^d$ has $\textup{EQP}$ coordinates, $g,h:\N\rightarrow \Z\in \textup{EQP}$, and $h$ is eventually nonzero.

Then there is some $S'_t\subseteq \Z^{d'}$, affine reducible to $S_t$, such that $S'_t\subseteq \Z^{d'}$ is defined by an $\mathcal{L}_{\textup{EQP}}$ formula that is a conjunction of atomic formulas of the form $\mathbf f(t)\cdot \mathbf{x} \le g(t)$, where $\mathbf f:\N\rightarrow \Z^d$ has $\textup{EQP}$ coordinates and $g(t):\N\rightarrow \Z$ is in $\textup{EQP}$.
\end{lem}

\begin{proof}
Let $R_t\subseteq\Z^d$ be the set satisfying the conjunction of the $D_{h(t)}$ terms. $R_t$ can be represented by:
\[\{\mathbf u_0(t)+\lambda_1 \mathbf u_1(t)+\cdots+\lambda_r \mathbf u_r(t):\ \lambda_j\in\Z\},\]
where $\mathbf u_1(t),\ldots,\mathbf u_r(t)$ are linearly independent, by Proposition \ref{prop:Hermite}. We take as our affine reduction $F:\Z^r\rightarrow\Z^d$, given by $F(\vec \lambda)=\mathbf u_0(t)+\lambda_1 \mathbf u_1(t)+\cdots+\lambda_r \mathbf u_r(t)$. Since the $\mathbf u_j$ are linearly independent, this defines a bijection $\Z^r\rightarrow R_t$. An element $\mathbf x\in S_t$ satisfying a linear inequality $\mathbf f(t)\cdot \mathbf{x} \le g(t)$ is equivalent to $\vec \lambda=F^{-1}(\mathbf x)$ satisfying
\[\mathbf f(t)\cdot \big(\mathbf u_0(t)+\lambda_1 \mathbf u_1(t)+\cdots+\lambda_r \mathbf u_r(t)\big) \le g(t),\]
which simplifies to a linear inequality of the desired form.
\end{proof}

Now we have $S_t$ written as the set of integer points in a ``parametric polyhedron'', that is, defined by linear inequalities that depend on $t$, as above. If this polyhedron is bounded, then it immediately follows from \cite{ChLiSam} that $\abs{S_t}\in \textup{EQP}$. We must detect when $S_t$ is unbounded. To do this, it is helpful to switch to an ``internal'' representation of the parametric polyhedron, as opposed to this ``external'' representation as a solution set to linear inequalities, cf. Example \ref{Ex:step5c}. That is, we want to write the polyhedron as
\[\conv\{\mathbf x_1,\ldots, \mathbf x_q\}+\cone\{\mathbf y_1,\ldots,\mathbf y_r\}+\spanOp\{\mathbf z_1,\ldots,\mathbf z_s\},\]
where
\begin{align*}
\mathbf x_i&\in\Q^d,\ \mathbf y_j,\mathbf z_k\in\Z^d-\{ 0\},\\
A+B&=\{a+b:\ a\in A,\ b\in B\},\\
\conv\{\mathbf x_1,\ldots, \mathbf x_q\}&=\{\lambda_1\mathbf x_1+\cdots+\lambda_q\mathbf x_q:\ \lambda_i\in\R,\ \lambda_i\ge 0,\ \lambda_1+\cdots+\lambda_r=1\},\\
\cone\{\mathbf y_1,\ldots,\mathbf  y_r\}&=\{\lambda_1\mathbf y_1+\cdots+\lambda_r\mathbf y_r:\ \lambda_i\in\R,\ \lambda_i\ge 0\},\\
\spanOp\{\mathbf z_1,\ldots, \mathbf z_s\}&=\{\lambda_1\mathbf z_1+\cdots+\lambda_s\mathbf z_s:\ \lambda_i\in\R\},
\end{align*}
and $\cone\{\mathbf y_1,\ldots,\mathbf y_r\}$ contains no lines. We follow the convention that $\conv\{\emptyset\}=\emptyset$, $\cone\{\emptyset\}=\{\mathbf 0\}$, and $\spanOp\{\emptyset\}=\{\mathbf 0\}$, so that this representation may include the empty polyhedron. Also note that the polyhedron is bounded exactly when  $r=s=0$. We have the following proposition:

\begin{prop}
\label{final_reduction}
Suppose $S_t\subseteq \Z^{d}$ is defined by
\[\bigwedge_{j=1}^n \mathbf f_j(t)\cdot\mathbf x \le g_j(t),\]
where $\mathbf f_j:\N\rightarrow \Z^d$ has $\textup{EQP}$ coordinates and $g_j:\N\rightarrow \Z$ in $\textup{EQP}$. Then there exists $N$ and a period $m$ such that, for fixed $i$ with $0\le i\le m$, there exist $\mathbf x_1,\ldots,\mathbf x_q:\N\rightarrow\Q^d$ and $\mathbf y_1,\ldots,\mathbf y_r,\mathbf z_1\ldots,\mathbf z_s:\N\rightarrow\Z^d$, with $\mathbf x_j$ rational functions in $\Q(t)$ and $\mathbf y_j,\mathbf z_k$ nonzero polynomials in $\Z[t]$, such that, for all $t\in\N$  with $t\ge N$ and $t\equiv i\bmod m$, we have that $S_t$ is the set of integer points in
\[\conv\{\mathbf x_1(t),\ldots, \mathbf x_q(t)\}+\cone\{\mathbf y_1(t),\ldots,\mathbf y_r(t)\}+\spanOp\{\mathbf z_1(t),\ldots,\mathbf z_s(t)\},\]
where $\cone\{\mathbf y_1(t),\ldots,\mathbf y_r(t)\}$ contains no lines.
\end{prop}

\begin{proof}
We may assume we are looking at a residue class $i\bmod m$ such that $g_j$ and the coordinates of $\mathbf f_j$ are in $\Q[t]$. Let $F$ be the $n\times d$ matrix whose $j^{\text{th}}$ row is $f_j$, and let $\mathbf g$ be the $n\times 1$ column vector whose $j^{\text{th}}$ entry is $g_j$. Regarding $\mathbf x\in\Z^d$ as a column vector, $S_t$ is the set integer solutions to $F\cdot\mathbf x\le \mathbf g$. Note that we may perform standard row reduction on $F$, in the field of rational functions $\Q(t)$. This suffices to compute the desired representation, using Theorem 8.5 of \cite{Sch}, as follows:
\begin{itemize}
\item For the convex hull, take any maximally linear independent subset of rows, $F'$, of $F$ and their corresponding $\mathbf g'$. (Note that any given subset of rows of $F$ is either eventually always linearly dependent or eventually always linearly independent, since this can be tested by computing determinants of minors of $F$, which are polynomial functions in $t$.) Take any solution $\mathbf x(t)$ to $F'\cdot\mathbf x = \mathbf g'$; $\mathbf x$ may be computed using row reduction and will have entries in $\Q(t)$. If $\mathbf x(t)$ eventually satisfies $F\cdot\mathbf x\le \mathbf g$, then
include it as one of the $\mathbf x_j(t)$.  Doing this for all maximally linearly independent set of rows of $F$ yields $\conv\{\mathbf x_1(t),\ldots, \mathbf x_q(t)\}$.
\item For the cone, take any subset of rows, $F'$, and a distinct row, $\mathbf f$, such that $F'$ and $\mathbf f$ together are a maximally linear independent subset of rows of $F$. Take any solution $\mathbf y(t)$ to
\[\left[\begin{matrix} F'\\ \mathbf f\end{matrix}\right]\mathbf y=\left[\begin{matrix}\mathbf 0 \\ -1\end{matrix}\right].\]
This may again be computed using row reduction and will have entries in $\Q(t)$. If $\mathbf y(t)$ eventually satisfies $F\cdot \mathbf y\le \mathbf 0$, then include $p(t)\mathbf y(t)$ as one of the $\mathbf y_j(t)$, where $p\in\Z[t]$ is eventually positive and makes $p\mathbf y\in\Z[t]$. Doing this for all such $F'$ and $\mathbf f$ yields
$\cone\{\mathbf y_1(t),\ldots,\mathbf y_r(t)\}$.
\item Use row reduction to find a basis for $F\cdot\mathbf z = \mathbf 0$. Clearing fractions to make each $\mathbf z_j(t)\in\Z[t]$, this yields $\spanOp\{\mathbf z_1(t),\ldots,\mathbf z_s(t)\}$.
\qedhere
\end{itemize}
\end{proof}

\subsection{Step 6}
In the output of Proposition~\ref{final_reduction}, any polyhedron $P$ with $r=s=0$ is bounded, and we may apply the results from \cite{ChLiSam} and from \cite{Shen15b}. The following lemma helps us deal with unbounded polyhedra.

\begin{lem} \label{auxiliaryQ}
Suppose we have a polyhedron, $P\subseteq\R^d$, given by
\[P=\conv\{\mathbf x_1,\ldots, \mathbf x_q\}+\cone\{\mathbf y_1,\ldots,\mathbf y_r\}+\spanOp\{\mathbf z_1,\ldots,\mathbf z_s\},\]
with $\mathbf x_i\in\Q^d$, $\mathbf y_j,\mathbf z_k\in\Z^d-\{0\}$. 
Let
\[Q=\conv\{\mathbf x_1,\ldots, \mathbf x_q\}+\zono\{\mathbf y_1,\ldots,\mathbf y_r,\mathbf z_1,\ldots,\mathbf z_s\},\]
where $\zono\{\mathbf w_1,\ldots,\mathbf w_u\}=\{\lambda_1\mathbf w_1+\cdots+\lambda_u\mathbf w_u:\ \lambda_i\in\R,\ 0\le\lambda_i\le 1\}$ is the zonotope (Minkowski sum of line segments) generated by the $\vec w_i$.

Assume either $r\ge 1$ or $s\ge 1$. Then
\begin{itemize}
\item if $Q$ contains an integer point, then $P$ contains an infinite number of integer points, and
\item if $Q$ contains no integer points, then $P$ contains no integer points.
\end{itemize}

Furthermore, $Q$ is a bounded polyhedron, given by the convex hull of the following points:
for each $W\subseteq\{\mathbf y_1,\ldots,\mathbf y_s,\mathbf z_1,\ldots,\mathbf z_t\}$ and $i$ with $1\le i\le q$, take the point
\[\mathbf x_i+\sum_{\mathbf w\in W}\mathbf  w.\] 
\end{lem}

\begin{proof}
Without loss of generality, $r\ge 1$. If $Q$ contains an integer point, $\mathbf x$, then certainly $\mathbf x\in P$, since $Q\subseteq P$. But also $\mathbf x+k\mathbf y_1\in P$, for all $k\in \N$, so $P$ contains an infinite number of integer points.

We prove the contrapositive of the second bullet point. Suppose $P$ contains an integer point
\[\sum_i\lambda_i\mathbf x_i+\sum_j\mu_j\mathbf y_j+\sum_k\nu_k\mathbf z_k,\]
with $\lambda_i,\mu_k\ge 0$, $\sum_i\lambda_i=1$. Then $Q$ contains the integer point
\[\sum_i\lambda_i\mathbf x_i+\sum_j[\mu_j]\mathbf y_j+\sum_k[\nu_k]\mathbf z_k,\]
where $[\eta]$ is the unique real number in $[0,1)$ such that $\eta-[\eta]$ is an integer.

The representation of $Q$ as a convex hull follows from the fact that $Q$ is the Minkowski sum of $\conv\{\mathbf x_1,\ldots, \mathbf x_q\}$, the line segments $[0,\vec y_i]$, and the line segments $[0,\vec z_i]$, and the standard observation that each vertex of a Minkowski sum can be written as the sum of vertices of the summand polytopes.
\end{proof}

We now have all the ingredients necessary to prove Properties 1, 2, and 3. Using all of the results of this section up to Proposition~\ref{final_reduction}, and noting that all of the properties only need to be shown for all sufficiently large $t$ congruent to some $i$ modulo a fixed number $m$, it suffices to show the following.       

\begin{prop}
\label{properties123}
Suppose $S_t$ is the set of integer points in \[\conv\{\mathbf x_1(t),\ldots, \mathbf x_q(t)\}+\cone\{\mathbf y_1(t),\ldots,\mathbf y_r(t)\}+\spanOp\{\mathbf z_1(t),\ldots,\mathbf z_s(t)\},\] as in Proposition~\ref{final_reduction}. Then 
\begin{enumerate}
\item \label{item_nonempty} The set of $t$ such that $S_t$ is nonempty is eventually periodic.
\item \label{item_finite} The set of $t$ such that $S_t$ has finite cardinality is eventually periodic.
\item \label{item_EQP} There exists an EQP $g:\N\rightarrow\N$ such that, if $S_t$ has finite cardinality, then $g(t)=\abs{S_t}$.
\item \label{item_coords} There exists a function $\x:\N\rightarrow\Z^d$, whose coordinate functions are EQPs, such that, if $S_t$ is non\-emp\-ty, then $\x(t)\in S_t$. 
\end{enumerate}
\end{prop}

\begin{proof}
Note that Woods \cite{woods1} shows that these follow from the EQP behavior of the generating function, which we discuss and prove in Section \ref{sec:pl} (at least for $S\subseteq\N^d$, a case we could reduce to). We therefore only present a brief proof here, to demonstrate that generating functions are unnecessary.

First consider the case that $r=s=0$; i.e. $S_t$ is the set of integer points in the parametric polytope 
 \[\conv\{\mathbf x_1(t),\ldots, \mathbf x_q(t)\}. \]
Then $S_t$ is finite for all $t$ so (\ref{item_finite}) is immediate. Furthermore, (\ref{item_EQP}) is the main result of Chen-Li-Sam \cite[Theorem 2.1]{ChLiSam}: the idea of the proof is that writing each variable ``base $t$''  (as in Step 3) reduces this to a problem in classic Ehrhart Theory  (see Example \ref{Ex:Ehrhart}). 

Next, (\ref{item_coords}) is effectively the main result of Shen \cite[Corollary 3.2.2]{Shen15b}, which proves that, for any $\vec c \in\Z^d$, the maximum value of $\vec c\cdot \vec x$ is EQP (in fact, it proves a slightly more general statement). To prove (\ref{item_coords}), we might demonstrate that an \emph{argmax} (an $\x$ maximizing this linear function $\vec c\cdot \vec x$) has EQP coordinates. While \cite{Shen15b} demonstrates only that the \emph{max} value of $\vec c\cdot \vec x$ is an EQP,  it is clear that the proof also works for the argmax. The key idea is that, after performing the ``base $t$'' reduction, such an $\x$ must be within a fixed distance to one of the facets of the parametric polytope \cite[Proposition 4.1]{Shen15b}; therefore, by searching all parallel hyperplanes within that fixed distance of a facet, we reduce the dimension of the problem inductively.

To instead derive (\ref{item_coords}) directly from the statement of Shen \cite[Corollary 3.2.2]{Shen15b}, we may do the following: First take, $c = e_1$ to get a maximum value for the $x_1$ coordinate, say $a_1$. Then, substitute the value $a_1$ in for $x_1$ into the definition of our polyhedron, reducing the dimension by one; we now have the cross-section of the original polyhedron at $x_1=a_1$. Then take $c=e_2$ to get the maximum in the $x_2$ coordinate, and repeat.

Finally, (\ref{item_nonempty}) follows from the other properties as proved in \cite{woods1}. For example, (\ref{item_EQP}) yields that $\abs{S_t}$ is an EQP, and the only way to have an infinite number of zeros in this EQP is for a constituent polynomial to be identically zero. Alternatively, (\ref{item_nonempty}) also follows from the proofs in \cite{Shen15b}.

Now suppose that either $r \geq 1$ or $s \geq 1$. Again let 
\[Q=\conv\{\mathbf x_1,\ldots, \mathbf x_q\}+\zono\{\mathbf y_1,\ldots,\mathbf y_r,\mathbf z_1,\ldots,\mathbf z_s\}.\]

By Lemma~\ref{auxiliaryQ}, $Q$ is a parametric polytope so we can apply the previous case to the set $S'_t := Q \cap \Z^d$. In particular, the set of $t$ for which $S'_t$ is nonempty is eventually periodic. But since (again by Lemma \ref{auxiliaryQ}) $S_t$ is finite iff $S_t$ is empty iff $S'_t$ is empty, (\ref{item_nonempty}), (\ref{item_finite}), and (\ref{item_EQP}) immediately follow. Also, (\ref{item_coords}) holds for $S'_t$ by the previous case. Since $Q \subseteq P$ and $S_t$ is empty iff $S'_t$ is empty, we see that (\ref{item_coords}) also holds for $S_t$ by simply choosing a point in $S'_t$.
\end{proof}

\section{Generating Functions}
\label{sec:pl}
Given a set $S\subseteq\N^d$, we define its generating function to be
\[\sum_{\vec s\in S}\x^{\vec s}=\sum_{\vec s\in S}x_1^{s_1}x_2^{s_2}\cdots x_d^{s_d}.\]
 These generating functions can often be simplified to a rational function.

\begin{example}
The generating function for $S=\N$ is
\[\sum_{s\in\N} x^s=1+x+x^2+\cdots=\frac{1}{1-x},\]
and the infinite sum is convergent for $\abs{x}<1$.
\end{example}

Now let $S_t$ be a parametric Presburger family. For a fixed $t\in\N$, $S_t$ has a generating function, and Property 4 examines what happens to this generating function as $t$ changes.

\begin{example}
\label{ex:tri_gf}For the triangle $P$ from Example \ref{ex:triangle} with vertices $(0,0)$, $\left(\frac{1}{2},0\right)$, and $\left(\frac{1}{2},\frac{1}{2}\right)$, we have the generating function
\begin{align*}
\sum_{\vec s\in tP\cap\Z^d} \x^{\vec s}&=\left(1+x_1+x_1^2+\cdots+x_1^{\floor{t/2}}\right)\\
&\phantom{mmm}+\left(x_1+x_1^2+\cdots+x_1^{\floor{t/2}}\right)x_2 +\cdots +\left(x_1^{\floor{t/2}}\right)x_2^{\floor{t/2}}\\
&=\frac{1}{(1-x_1)(1-x_1x_2)}-\frac{x_1^{\floor{t/2}+1}}{(1-x_1)(1-x_2)}+\frac{x_1^{\floor{t/2}+1}x_2^{\floor{t/2}+2}}{(1-x_2)(1-x_1x_2)}\\
&=\frac{1-x_2-(1-x_1x_2)x_1^{\floor{t/2}+1}+(1-x_1) x_1^{\floor{t/2}+1}x_2^{\floor{t/2}+2}}{(1-x_1)(1-x_2)(1-x_1x_2)}
\end{align*}
(the second equality can be verified directly by expanding the fractions as products of geometric series).
\end{example}

Note that the above example satisfies Property 4.

Generating functions are powerful tools for understanding the structure of $S_t$. For example, given the generating function, we can count the number of integer points in $S_t$ (if finite) by substituting in $\x=(1,\ldots,1)$.

\begin{example}
Substituting $x_1=x_2=1$ into the generating function for Example \ref{ex:tri_gf}, we see that $(1,1)$ is a zero of the numerator and denominator of this rational function. Fortunately, applying L'H\^opital's rule to find the limit as $x_1$ and $x_2$ approach 1 will work, and it is evident that the differentiation involved in L'H\^opital's rule will yield a quasi-polynomial in $t$ as the result; careful calculation will show that it matches Example \ref{ex:triangle}.
\end{example}

\begin{remark}\label{rmk:lines}
To prove Property 4 for parametric Presburger families, we will assume from now on that $S_t\subseteq\N^d$. We must check that the EQP-affine reductions in Steps 2, 3, and 5 above result in a family $S'_t$ which is still a subset of $\N^{d'}$ for some $d'$ (rather than simply a subset of $\Z^{d'}$), but this is immediate upon inspection. Steps 1 and 4 do not cause any problems either since they give logically equivalent formulas which define the same parametric Presburger family.
\end{remark}

While one may instead define generating functions for subsets of $\Z^d$, this may lose information about the sets:

\begin{example}
If $S=\Z$, then the generating function for $S$ is
\[\cdots+x^{-2}+x^{-1}+1+x^1+x^2+\cdots.\]
If you regard this as an infinite series, then it doesn't converge on any open neighborhood. If you regard it as a formal power series, then it simplifies to
\[(\cdots+x^{-2}+x^{-1})+(1+x^1+x^2+\cdots)=\frac{x^{-1}}{1-x^{-1}}+\frac{1}{1-x}=-\frac{1}{1-x}+\frac{1}{1-x}=0\]
which is also the generating function for the empty set.
\end{example}

In the statement of Property 4, we need $\vec b_{ij}(t)$ to be lexicographically positive rather than simply in $\N^d\setminus\{0\}$ for examples like the following:

\begin{example}
\label{ex:negs}
Let $S_t=\big\{(x,y)\in\N^2:\ x+y=t\big\}$. Then the generating function for $S_t$ is
\[y^{t}+xy^{t-1}+\cdots+x^{t}=\frac{y^{t}-x^{t+1}y^{-1}}{1-xy^{-1}}.\]
\end{example}

\begin{remark}\label{rmk:lexpos}
If $\vec b$ is lexicographically negative, then we may instead write
\[\frac{1}{1-\vec x^{\vec b}}=\frac{1}{1-\vec x^{\vec b}}\cdot\frac{-\vec x^{-\vec b}}{-\vec x^{-\vec b}}=\frac{-\vec x^{-\vec b}}{1-\vec x^{-\vec b}}\]
with $-\vec b$ lexicographically positive.
\end{remark}

\begin{remark}\label{rmk:lexpos2}
Having $\vec b$ lexicographically positive guarantees that $1/(1-\vec x^{\vec b})=1+\vec x^{\vec b}+\vec x^{2\vec b}+\cdots$ is the Laurent series convergent on a neighborhood of $\vec a \doteq (e^{-\varepsilon},e^{-\varepsilon^2},\ldots,e^{-\varepsilon^d})$, for sufficiently small $\varepsilon$, as follows: Let $b_i>0$ be the first nonzero coordinate of $\vec b$. Then
\[\ln\left(\vec a ^ {\vec b}\right)=-\varepsilon b_1-\varepsilon^2 b_2-\cdots -\varepsilon^d b_d=-\varepsilon^i b_i-\cdots-\varepsilon^d b_d.\]
For sufficiently small $\varepsilon$, the  $-\varepsilon^i b_i$ term dominates, and $\ln(\vec a ^ {\vec b})<0$. This implies that $\abs{\vec a ^ {\vec b}}<1$, and indeed $1+\vec x^{\vec b}+\vec x^{2\vec b}+\cdots$ does converge on a neighborhood of $\vec a$.
\end{remark}

Now we prove Property 4, for parametric Presburger families. The hardest work has already been done:
\begin{itemize}
\item At the end of Step 4, we had that $S_t$ was affine reducible to a set $S'_t$, described by a quantifier free formula in parametric Presburger arithmetic.
\item By \cite[Theorem 3.5(a)]{woods1}, any set described by a quantifier-free formula has Property 4.
\end{itemize}

The only remaining step is to check that Property 4 is preserved by affine reductions. First a short lemma:

\begin{lem} \label{gftransform}
  Let $S \subseteq \N^d$ and $S' \subseteq \N^e$ be sets such that there is a $\Z$-affine linear function $F: \Z^e \to \Z^d$ that restricts to a bijection from $S'$ to $S$. Let 
  $F(\x) = A \x + \mathbf{c}$, where $A=(a_{ij})$. Define monomials $m_i = w_1^{a_{1i}} \cdots w_d^{a_{di}}$ for $i = 1, \dots, e$.  Let $g(z_1, \ldots, z_e) = g(\z)$ be the generating function for $S'$. Then the generating function for $S$ is $\w^{\mathbf{c}} g(m_1, \ldots, m_e),$ where $\w := (w_1, \dots, w_d)$.
\end{lem}

\begin{proof}
  Let $h(w_1, \dots, w_d) =: h(\w)$ be the generating function for $S$. Then
  \begin{align*}
    h(\w) & = & \sum_{\y \in S} \w^\y \\
    & = & \sum_{\x \in S'} \w^{A \x + \mathbf{c}} \\
    & = & \w^{\mathbf{c}} \sum_{\x \in S'} w_1^{a_{11}x_1 + \ldots + a_{1e}x_e} \cdots w_d^{a_{d1}x_1 + \ldots + a_{de}x_e} \\
    & = & \w^{\mathbf{c}} \sum_{\x \in S'} \left(w_1^{a_{11}} \cdots w_d^{a_{d1}}\right)^{x_1} \cdots \left(w_1^{a_{1e}} \cdots w_d^{a_{de}}\right)^{x_e} \\
      & = & \w^{\mathbf{c}} \sum_{\x \in S'} m_1^{x_1} \cdots m_e^{x_e} \\
      & = & \w^{\mathbf{c}} g(m_1, \ldots, m_e)
  \end{align*}
where the second equality follows because $F$ restricts to a bijection from $S'$ to $S$. 
\end{proof}

\begin{prop} \label{preserve4}
Suppose $S'_t \subseteq \N^{d'}$ is EQP-affine reducible to $S_t \subseteq \N^d$. If $S'_t$ has Property 4, then so does $S_t$.
\end{prop}

\begin{proof}
  First consider the case that the affine reduction is defined by polynomials and that there is no periodicity in the generating function $g(\z, t)$ for $S'_t$. That is, $F(\x) = A(t) \x + \mathbf{c}(t)$ where $A(t) \in \Z[t]^{d \times e}$ and $\mathbf{c}(t) \in \Z[t]^d$, and
  \[ g(\z,t) := \sum_{\x\in S'_t}\z^\x = \frac{\sum_{j=1}^{n_i}\alpha_j \z^{\mathbf{q_j(t)}}}{(1-\z^{\mathbf{b_1(t)}}) \cdots (1-\z^{\mathbf{b_k}(t)})},\]
  where the coordinate functions of $\mathbf{q}_j$ and $\mathbf{b_j}$ are polynomials with the $\mathbf{b_j}$ lexicographically positive. 

  Fix notation as in Lemma \ref{gftransform} and additionally define $\mathbf{m}(t) = (m_1(t), \ldots m_e(t))$, where $m_i(t) = w_1^{a_{1i}(t)} \cdots w_d^{a_{di}(t)}$ and $w_1, \ldots, w_d$ are variables. Applying Lemma~\ref{gftransform} for each value of $t$, we see that the generating function $h(\w,t)$ of $S_t$ is given by
  \begin{align*} h(\w, t) & = & \w^{\mathbf{c}(t)} g(m_1(t), \dots, m_e(t)) \\
    & = & \w^{\mathbf{c}(t)} \frac{\sum_{j=1}^{n_i}\alpha_j \mathbf{m}(t)^{\mathbf{q_j(t)}}}{(1-\mathbf{m}(t)^{\mathbf{b_1(t)}}) \cdots (1-\mathbf{m}(t)^{\mathbf{b_k}(t)})}.  
  \end{align*}
  Writing $\mathbf{q_j(t)} = (q_j^1(t), \ldots, q_j^e(t))$, we have that
  \begin{align*}
    \mathbf{m}(t)^{q_j(t)} & = & \left(w_1^{a_{11}(t)} \cdots w_d^{a_{d1}(t)}  \right)^{q_j^1(t)} \cdots \left(w_1^{a_{1e}(t)} \cdots w_d^{a_{de}(t)}  \right)^{q_j^e(t)} \\
        & = & w_1^{a_{11}(t)q_j^1(t) + \ldots + a_{1e}(t)q_j^e(t)} \cdots w_d^{a_{d1}(t)q^1_j(t) + \ldots + a_{de}(t)q_j^e(t)} 
  \end{align*}
  so the exponents of $w_1, \ldots, w_d$ in the numerator are polynomial functions of $t$, and remain so after multiplying in the factor of $\w^\mathbf{c}(t)$ outside the sum. Similarly, the exponents of $w_1, \ldots, w_d$ in each factor in the denominator are polynomials in $t$. If any of these exponents are (eventually) lexicographically negative, apply Remark \ref{rmk:lexpos} to make them eventually lexicographically positive.
  
  In the general case, let $t_0$ be chosen sufficiently large so that the affine reduction is defined by quasi-polynomials for all $t > t_0$ and also the function $g$ has the desired form as a rational function for all $t > t_0$. Let $D$ be the least common multiple of all periods occurring in either the affine reduction or the generating function $g$. Then the generating function $h$ has the desired form with period $D$ for all $t > t_0$.  
\end{proof}



\bibliographystyle{amsplain}


\begin{dajauthors}
\begin{authorinfo}[tbog]
  Tristram Bogart\\
  Associate Professor\\
  Universidad de los Andes\\
  Bogot\'a, Colombia\\
  tc.bogart22\imageat{}uniandes\imagedot{}edu\imagedot{}co \\
  \url{http://wwwprof.uniandes.edu.co/~tc.bogart22/}
\end{authorinfo}
\begin{authorinfo}[jgoo]
  John Goodrick\\
  Associate Professor\\
  Universidad de los Andes\\
  Bogot\'a, Colombia\\
  jr.goodrick427\imageat{}uniandes\imagedot{}edu\imagedot{}co \\
  \url{http://matematicas.uniandes.edu.co/~goodrick/}
\end{authorinfo}
\begin{authorinfo}[kwoo]
  Kevin Woods\\
  Associate Professor\\
  Oberlin College\\
  Oberlin, Ohio, USA\\
  Kevin.Woods\imageat{}oberlin\imagedot{}edu\\
  \url{http://www.oberlin.edu/faculty/kwoods/}
\end{authorinfo}
\end{dajauthors}

\end{document}